\documentclass[11pt]{amsart}

\usepackage{amsmath}
\usepackage{amssymb}
\usepackage{amscd}
\usepackage{color}
\usepackage{hyperref}
\usepackage{mathrsfs}
\usepackage{mathtools}
\usepackage{eucal}
\usepackage{ytableau}
\usepackage{comment}
\usepackage{enumitem}
\usepackage{upgreek}

\usepackage[normalem]{ulem}

\usepackage{xy}
\xyoption{all}

\newcommand{\nc}{\newcommand}

\renewcommand{\AA}{{\mathbb{A}}}
\nc{\CC}{{\mathbb{C}}}
\nc{\LL}{{\mathbb{L}}}
\nc{\RR}{{\mathbb{R}}}
\renewcommand{\P}{{\mathbb{P}}}
\nc{\PP}{{\mathbb{P}}}
\nc{\OO}{{\mathbb{O}}}

\nc{\QQ}{{\mathbb{Q}}}
\nc{\ZZ}{{\mathbb{Z}}}

\nc{\cA}{{\mathcal{A}}}
\nc{\cB}{{\mathcal{C}\!\ell}}
\nc{\rcB}{\operatorname{\mathbf{Cl}}}
\nc{\cC}{{\mathcal{C}}}
\nc{\cD}{{\mathcal{D}}}
\nc{\cE}{{\mathcal{E}}}
\nc{\cF}{{\mathcal{F}}}
\nc{\cG}{{\mathcal{G}}}
\nc{\cH}{{\mathcal{H}}}
\nc{\cI}{{\mathcal{I}}}
\nc{\cJ}{{\mathcal{J}}}
\nc{\cK}{{\mathcal{K}}}
\nc{\cL}{{\mathcal{L}}}
\nc{\cM}{{\mathcal{M}}}
\nc{\cN}{{\mathcal{N}}}
\nc{\cO}{{\mathcal{O}}}
\nc{\cP}{{\mathcal{P}}}
\nc{\cQ}{{\mathcal{Q}}}
\nc{\cR}{{\mathcal{R}}}
\nc{\cS}{{\mathcal{S}}}
\nc{\cT}{{\mathcal{T}}}
\nc{\cU}{{\mathcal{U}}}
\nc{\cV}{{\mathcal{V}}}
\nc{\cW}{{\mathcal{W}}}
\nc{\cX}{{\mathcal{X}}}
\nc{\cY}{{\mathcal{Y}}}
\nc{\cZ}{{\mathcal{Z}}}

\nc{\bA}{{\mathbf{A}}}
\nc{\bB}{{\mathbf{B}}}
\nc{\bC}{{\mathbf{C}}}
\nc{\bD}{{\mathbf{D}}}
\nc{\bE}{{\mathbf{E}}}
\nc{\bF}{{\mathbf{F}}}
\nc{\bG}{{\mathbf{G}}}
\nc{\bH}{{\mathbf{H}}}
\nc{\bI}{{\mathbf{I}}}
\nc{\bJ}{{\mathbf{J}}}
\nc{\bK}{{\mathbf{K}}}
\nc{\bL}{{\mathbf{L}}}
\nc{\bM}{{\mathbf{M}}}
\nc{\bN}{{\mathbf{N}}}
\nc{\bO}{{\mathbf{O}}}
\nc{\bP}{{\mathbf{P}}}
\nc{\bQ}{{\mathbf{Q}}}
\nc{\bR}{{\mathbf{R}}}
\nc{\bS}{{\mathbf{S}}}
\nc{\bT}{{\mathbf{T}}}
\nc{\bU}{{\mathbf{U}}}
\nc{\bV}{{\mathbf{V}}}
\nc{\bW}{{\mathbf{W}}}
\nc{\bX}{{\mathbf{X}}}
\nc{\bY}{{\mathbf{Y}}}
\nc{\bZ}{{\mathbf{Z}}}

\nc{\ba}{{\mathbf{a}}}
\nc{\bb}{{\mathbf{b}}}
\nc{\bc}{{\mathbf{c}}}
\nc{\bd}{{\mathbf{d}}}
\nc{\be}{{\mathbf{e}}}
\nc{\bg}{{\mathbf{g}}}
\nc{\bh}{{\mathbf{h}}}
\nc{\bi}{{\mathbf{i}}}
\nc{\bj}{{\mathbf{j}}}
\nc{\bk}{{\mathbf{k}}}
\nc{\bl}{{\mathbf{l}}}
\nc{\bm}{{\mathbf{m}}}
\nc{\bn}{{\mathbf{n}}}
\nc{\bo}{{\mathbf{o}}}
\nc{\bp}{{\mathbf{p}}}
\nc{\bq}{{\mathbf{q}}}
\nc{\br}{{\mathbf{r}}}
\nc{\bs}{{\mathbf{s}}}
\nc{\bt}{{\mathbf{t}}}
\nc{\bu}{{\mathbf{u}}}
\nc{\bv}{{\mathbf{v}}}
\nc{\bw}{{\mathbf{w}}}
\nc{\bx}{{\mathbf{x}}}
\nc{\by}{{\mathbf{y}}}
\nc{\bz}{{\mathbf{z}}}

\nc{\fA}{{\mathfrak{A}}}
\nc{\fB}{{\mathfrak{B}}}
\nc{\fC}{{\mathfrak{C}}}
\nc{\fD}{{\mathfrak{D}}}
\nc{\fE}{{\mathfrak{E}}}
\nc{\fF}{{\mathfrak{F}}}
\nc{\fG}{{\mathfrak{G}}}
\nc{\fH}{{\mathfrak{H}}}
\nc{\fI}{{\mathfrak{I}}}
\nc{\fJ}{{\mathfrak{J}}}
\nc{\fK}{{\mathfrak{K}}}
\nc{\fL}{{\mathfrak{L}}}
\nc{\fM}{{\mathfrak{M}}}
\nc{\fN}{{\mathfrak{N}}}
\nc{\fO}{{\mathfrak{O}}}
\nc{\fP}{{\mathfrak{P}}}
\nc{\fQ}{{\mathfrak{Q}}}
\nc{\fR}{{\mathfrak{R}}}
\nc{\fS}{{\mathfrak{S}}}
\nc{\fT}{{\mathfrak{T}}}
\nc{\fU}{{\mathfrak{U}}}
\nc{\fV}{{\mathfrak{V}}}
\nc{\fW}{{\mathfrak{W}}}
\nc{\fX}{{\mathfrak{X}}}
\nc{\fY}{{\mathfrak{Y}}}
\nc{\fZ}{{\mathfrak{Z}}}

\nc{\fa}{{\mathfrak{a}}}
\nc{\fb}{{\mathfrak{b}}}
\nc{\fc}{{\mathfrak{c}}}
\nc{\fd}{{\mathfrak{d}}}
\nc{\fe}{{\mathfrak{e}}}
\nc{\ff}{{\mathfrak{f}}}
\nc{\fg}{{\mathfrak{g}}}
\nc{\fh}{{\mathfrak{h}}}
\nc{\fj}{{\mathfrak{j}}}
\nc{\fk}{{\mathfrak{k}}}
\nc{\fl}{{\mathfrak{l}}}
\nc{\fm}{{\mathfrak{m}}}
\nc{\fn}{{\mathfrak{n}}}
\nc{\fo}{{\mathfrak{o}}}
\nc{\fp}{{\mathfrak{p}}}
\nc{\fq}{{\mathfrak{q}}}
\nc{\fr}{{\mathfrak{r}}}
\nc{\fs}{{\mathfrak{s}}}
\nc{\ft}{{\mathfrak{t}}}
\nc{\fu}{{\mathfrak{u}}}
\nc{\fv}{{\mathfrak{v}}}
\nc{\fw}{{\mathfrak{w}}}
\nc{\fx}{{\mathfrak{x}}}
\nc{\fy}{{\mathfrak{y}}}
\nc{\fz}{{\mathfrak{z}}}

\nc{\sA}{{\mathsf{A}}}
\nc{\sB}{{\mathsf{B}}}
\nc{\sC}{{\mathsf{C}}}
\nc{\sD}{{\mathsf{D}}}
\nc{\sE}{{\mathsf{E}}}
\nc{\sF}{{\mathsf{F}}}
\nc{\sG}{{\mathsf{G}}}
\nc{\sH}{{\mathsf{H}}}
\nc{\sI}{{\mathsf{I}}}
\nc{\sJ}{{\mathsf{J}}}
\nc{\sK}{{\mathsf{K}}}
\nc{\sL}{{\mathsf{L}}}
\nc{\sM}{{\mathsf{M}}}
\nc{\sN}{{\mathsf{N}}}
\nc{\sO}{{\mathsf{O}}}
\nc{\sP}{{\mathsf{P}}}
\nc{\sQ}{{\mathsf{Q}}}
\nc{\sR}{{\mathsf{R}}}
\nc{\sS}{{\mathsf{S}}}
\nc{\sT}{{\mathsf{T}}}
\nc{\sU}{{\mathsf{U}}}
\nc{\sV}{{\mathsf{V}}}
\nc{\sW}{{\mathsf{W}}}
\nc{\sX}{{\mathsf{X}}}
\nc{\sY}{{\mathsf{Y}}}
\nc{\sZ}{{\mathsf{Z}}}

\nc{\sa}{{\mathsf{a}}}
\nc{\sd}{{\mathsf{d}}}
\nc{\se}{{\mathsf{e}}}
\nc{\sg}{{\mathsf{g}}}
\nc{\sh}{{\mathsf{h}}}
\nc{\si}{{\mathsf{i}}}
\nc{\sj}{{\mathsf{j}}}
\nc{\sk}{{\mathsf{k}}}
\nc{\sm}{{\mathsf{m}}}
\nc{\sn}{{\mathsf{n}}}
\nc{\so}{{\mathsf{o}}}
\nc{\sq}{{\mathsf{q}}}
\nc{\sr}{{\mathsf{r}}}
\nc{\st}{{\mathsf{t}}}
\nc{\su}{{\mathsf{u}}}
\nc{\sv}{{\mathsf{v}}}
\nc{\sw}{{\mathsf{w}}}
\nc{\sx}{{\mathsf{x}}}
\nc{\sy}{{\mathsf{y}}}
\nc{\sz}{{\mathsf{z}}}

\nc{\oA}{{\overline{A}}}
\nc{\oB}{{\overline{B}}}
\nc{\oC}{{\overline{C}}}
\nc{\oD}{{\overline{D}}}
\nc{\oE}{{\overline{E}}}
\nc{\oF}{{\overline{F}}}
\nc{\oG}{{\overline{G}}}
\nc{\oH}{{\overline{H}}}
\nc{\oI}{{\overline{I}}}
\nc{\oJ}{{\overline{J}}}
\nc{\oK}{{\overline{K}}}
\nc{\oL}{{\overline{L}}}
\nc{\oM}{{\overline{M}}}
\nc{\oN}{{\overline{N}}}
\nc{\oO}{{\overline{O}}}
\nc{\oP}{{\overline{P}}}
\nc{\oQ}{{\overline{Q}}}
\nc{\oR}{{\overline{R}}}
\nc{\oS}{{\overline{S}}}
\nc{\oT}{{\overline{T}}}
\nc{\oU}{{\overline{U}}}
\nc{\oV}{{\overline{V}}}
\nc{\oW}{{\overline{W}}}
\nc{\oX}{{\overline{X}}}
\nc{\oY}{{\overline{Y}}}
\nc{\oZ}{{\overline{Z}}}

\nc{\oa}{{\overline{a}}}
\nc{\ob}{{\overline{b}}}
\nc{\oc}{{\overline{c}}}
\nc{\od}{{\overline{d}}}
\nc{\of}{{\overline{f}}}
\nc{\og}{{\overline{g}}}
\nc{\oh}{{\overline{h}}}
\nc{\oi}{{\overline{i}}}
\nc{\oj}{{\overline{j}}}
\nc{\ok}{{\overline{k}}}
\nc{\ol}{{\overline{l}}}
\nc{\om}{{\overline{m}}}
\nc{\on}{{\overline{n}}}
\nc{\oo}{{\overline{o}}}
\nc{\op}{{\overline{p}}}
\nc{\oq}{{\overline{q}}}
\nc{\os}{{\overline{s}}}
\nc{\ot}{{\overline{t}}}
\nc{\ou}{{\overline{u}}}
\nc{\ov}{{\overline{v}}}
\nc{\ow}{{\overline{w}}}
\nc{\ox}{{\overline{x}}}
\nc{\oy}{{\overline{y}}}
\nc{\oz}{{\overline{z}}}

\nc{\tA}{{\tilde{A}}}
\nc{\tB}{{\tilde{B}}}
\nc{\tC}{{\tilde{C}}}
\nc{\tD}{{\tilde{D}}}
\nc{\tE}{{\tilde{E}}}
\nc{\tF}{{\tilde{F}}}
\nc{\tG}{{\tilde{G}}}
\nc{\tH}{{\tilde{H}}}
\nc{\tI}{{\tilde{I}}}
\nc{\tJ}{{\tilde{J}}}
\nc{\tK}{{\tilde{K}}}
\nc{\tL}{{\tilde{L}}}
\nc{\tM}{{\tilde{M}}}
\nc{\tN}{{\tilde{N}}}
\nc{\tO}{{\tilde{O}}}
\nc{\tP}{{\tilde{P}}}
\nc{\tQ}{{\tilde{Q}}}
\nc{\tR}{{\tilde{R}}}
\nc{\tS}{{\tilde{S}}}
\nc{\tT}{{\tilde{T}}}
\nc{\tU}{{\tilde{U}}}
\nc{\tV}{{\tilde{V}}}
\nc{\tW}{{\tilde{W}}}
\nc{\tX}{{\tilde{X}}}
\nc{\tY}{{\tilde{Y}}}
\nc{\tZ}{{\tilde{Z}}}

\nc{\ta}{{\tilde{a}}}
\nc{\tb}{{\tilde{b}}}
\nc{\tc}{{\tilde{c}}}
\nc{\td}{{\tilde{d}}}
\nc{\te}{{\tilde{e}}}
\nc{\tf}{{\tilde{f}}}
\nc{\tg}{{\tilde{g}}}
\nc{\ti}{{\tilde{i}}}
\nc{\tj}{{\tilde{j}}}
\nc{\tk}{{\tilde{k}}}
\nc{\tl}{{\tilde{l}}}
\nc{\tm}{{\tilde{m}}}
\nc{\tn}{{\tilde{n}}}
\nc{\tp}{{\tilde{p}}}
\nc{\tq}{{\tilde{q}}}
\nc{\tr}{{\tilde{r}}}
\nc{\ts}{{\tilde{s}}}
\nc{\tu}{{\tilde{u}}}
\nc{\tv}{{\tilde{v}}}
\nc{\tw}{{\tilde{w}}}
\nc{\tx}{{\tilde{x}}}
\nc{\ty}{{\tilde{y}}}
\nc{\tz}{{\tilde{z}}}

\nc{\hA}{{\hat{A}}}
\nc{\hB}{{\hat{B}}}
\nc{\hC}{{\hat{C}}}
\nc{\hD}{{\hat{D}}}
\nc{\hE}{{\hat{E}}}
\nc{\hF}{{\hat{F}}}
\nc{\hG}{{\hat{G}}}
\nc{\hH}{{\hat{H}}}
\nc{\hI}{{\hat{I}}}
\nc{\hJ}{{\hat{J}}}
\nc{\hK}{{\hat{K}}}
\nc{\hL}{{\hat{L}}}
\nc{\hM}{{\hat{M}}}
\nc{\hN}{{\hat{N}}}
\nc{\hO}{{\hat{O}}}
\nc{\hP}{{\hat{P}}}
\nc{\hQ}{{\hat{Q}}}
\nc{\hR}{{\hat{R}}}
\nc{\hS}{{\hat{S}}}
\nc{\hT}{{\hat{T}}}
\nc{\hU}{{\hat{U}}}
\nc{\hV}{{\hat{V}}}
\nc{\hW}{{\hat{W}}}
\nc{\hX}{{\hat{X}}}
\nc{\hY}{{\hat{Y}}}
\nc{\hZ}{{\hat{Z}}}

\nc{\ha}{{\hat{a}}}
\nc{\hb}{{\hat{b}}}
\nc{\hc}{{\hat{c}}}
\nc{\hd}{{\hat{d}}}
\nc{\he}{{\hat{e}}}
\nc{\hf}{{\hat{f}}}
\nc{\hg}{{\hat{g}}}
\nc{\hh}{{\hat{h}}}
\nc{\hi}{{\hat{i}}}
\nc{\hj}{{\hat{j}}}
\nc{\hk}{{\hat{k}}}
\nc{\hl}{{\hat{l}}}
\nc{\hm}{{\hat{m}}}
\nc{\hn}{{\hat{n}}}
\nc{\ho}{{\hat{o}}}
\nc{\hp}{{\hat{p}}}
\nc{\hq}{{\hat{q}}}
\nc{\hr}{{\hat{r}}}
\nc{\hs}{{\hat{s}}}
\nc{\hu}{{\hat{u}}}
\nc{\hv}{{\hat{v}}}
\nc{\hw}{{\hat{w}}}
\nc{\hx}{{\hat{x}}}
\nc{\hy}{{\hat{y}}}
\nc{\hz}{{\hat{z}}}

\nc{\rA}{{\mathrm{A}}}
\nc{\rB}{{\mathrm{B}}}
\nc{\rD}{{\mathrm{D}}}
\nc{\rE}{{\mathrm{E}}}
\nc{\rH}{{\mathrm{H}}}
\nc{\rR}{{\mathrm{R}}}
\nc{\rT}{{\mathrm{T}}}

\nc{\eps}{\varepsilon}
\nc{\lan}{\big\langle}
\nc{\ran}{\big\rangle}
\nc{\kk}{\Bbbk}

\def\bw#1#2{{\textstyle{\bigwedge\hskip-0.9mm^{#1}}\hskip0.2mm{#2}}}

\DeclareMathOperator{\Hom}{\mathrm{Hom}}
\DeclareMathOperator{\Ext}{\mathrm{Ext}}

\DeclareMathOperator{\rd}{d}

\DeclareMathOperator{\cHom}{\mathcal{H}\mathit{om}}

\DeclareMathOperator{\cRHom}{\mathrm{R}\mathcal{H}\mathit{om}}

\DeclareMathOperator{\Spec}{\mathrm{Spec}}

\DeclareMathOperator{\Sing}{\mathrm{Sing}}

\DeclareMathOperator{\Sym}{\mathrm{Sym}}
\DeclareMathOperator{\Ker}{\mathrm{Ker}}

\DeclareMathOperator{\Ima}{\mathrm{Im}}

\DeclareMathOperator{\YD}{YD}
\DeclareMathOperator{\SYD}{SYD}

\DeclareMathOperator{\Gr}{\mathrm{Gr}}
\DeclareMathOperator{\OGr}{\mathrm{OGr}}

\DeclareMathOperator{\Gm}{\mathbb{G}_{\mathrm{m}}}
\DeclareMathOperator{\GL}{\mathrm{GL}}
\DeclareMathOperator{\tGL}{\widetilde{\mathrm{GL}}}

\DeclareMathOperator{\PGL}{\mathrm{PGL}}

\DeclareMathOperator{\SO}{\mathrm{SO}}
\DeclareMathOperator{\Spin}{\mathrm{Spin}}

\DeclareMathOperator{\id}{\mathrm{id}}

\DeclareMathOperator{\qgr}{qgr}
\DeclareMathOperator{\coh}{coh}

\nc{\Db}{{\mathbf{D}}}
\nc{\diag}{\mathrm{diag}}
\nc{\fSo}{\fS^\circ}

\nc{\DV}{W}

\theoremstyle{plain}

\newtheorem{theorem}{Theorem}[section]

\newtheorem{lemma}[theorem]{Lemma}
\newtheorem{proposition}[theorem]{Proposition}
\newtheorem{corollary}[theorem]{Corollary}

\theoremstyle{definition}

\newtheorem{definition}[theorem]{Definition}

\newtheorem{example}[theorem]{Example}

\theoremstyle{remark}

\newtheorem{remark}[theorem]{Remark}

\newenvironment{alenumerate}{\begin{enumerate}[label={\textup{(\alph*)}}]}{\end{enumerate}}

\title{Clifford spaces of empty intersections of quadrics}

\author{Alexander Kuznetsov}
\address{{\sloppy
\parbox{0.9\textwidth}{
Algebraic Geometry Section, Steklov Mathematical Institute of Russian Academy of Sciences,\\
8 Gubkin str., Moscow 119991 Russia
}\bigskip}}
\email{akuznet@mi-ras.ru}
\date{}

\begin{document}

\begin{abstract}
Given a linear space~$U \subset \Sym^2V^\vee$ of quadrics
in a projective space~$\P(V)$ whose intersection is empty,
we consider the corresponding {\sf Clifford space} ---
the projective space~$\P(U)$ endowed with the even part of Clifford algebras as a sheaf of algebras.
We show that the derived category of a Clifford space is generated by a full exceptional collection that extends to a 1-periodic helix
and the Clifford space is equivalent to the noncommutative projective spectrum of the corresponding graded algebra.

We discuss two special cases of Clifford spaces in more detail.

The first is the {\sf maximal Clifford space}, associated to the complete linear system~$U = \Sym^2V^\vee$ of quadrics.
It is homologically projectively dual to the second Veronese embedding of the projective space~$\P(V)$.
We show that the corresponding graded algebra is the maximal multiplicity-free direct sum
of all polynomial representations of~$\GL(V)$ and describe its dual $\rA_\infty$-algebra.

The second is a {\sf minimal Clifford space}, associated to a linear system of quadrics with~$\dim(U) = \dim(V)$.
We show that the corresponding graded algebra is a Koszul flat deformation of a polynomial algebra
and its dual algebra is a Frobenius flat deformation of an exterior algebra.
In particular, a minimal Clifford space is an example of a noncommutative projective space.
\end{abstract}

\maketitle

\section{Introduction}

Complete intersections of quadrics form a very important class of varieties studied in algebraic geometry,
see for instance the classical papers 
of Reid~\cite{Reid}, discussing a complete intersection of two quadrics (with a recent update in~\cite{CT}),
and Tyurin~\cite{Tyurin} and O'Grady~\cite{OG86} for the case of three quadrics.
The classical construction discussed in these papers 
relates the geometric properties of a complete intersection~$X_U \subset \P(V)$
of a $k$-dimensional linear space~$U \subset \Sym^2V^\vee$ of quadrics
in a projective space~$\P(V)$ of dimension~$n - 1$
to the degree~$n$ discriminant divisor~$\Delta \subset \P(U)$ (parameterizing singular quadrics in~$U$)
and the double covering~$\tilde\Delta \to \Delta$ if~$n$ is odd
or~$\widetilde\P(U) \to \P(U)$ if~$n$ is even.

A more general perspective was offered in~\cite{K08}, where a categorical version of this relation was discovered.
In particular, in the case where~$k < n/2$
(so that~$X_U$ is a Fano variety of index~$n - 2k$ and dimension~$n - k - 1$)
\cite[Theorem~5.5]{K08} provided a semiorthogonal decomposition
\begin{equation}
\label{eq:dbxw}
\Db(X_U) = \Big\langle \Db(\P(U), \cB_0(V)), \cO_{X_U}, \cO_{X_U}(1), \dots, \cO_{X_U}(n-2k-1) \Big\rangle,
\end{equation}
where the first component is the derived category of coherent sheaves of modules 
over the sheaf~$\cB_0(V)$ of even parts of Clifford algebras on~$\P(U)$ corresponding to the quadrics parameterized by~$U$,
and what stands to its right is an exceptional collection of line bundles of length~$n - 2k$.

The algebra~$\cB_0(V)$ is Azumaya on~$\P(U) \setminus \Delta$ if~$n$ is odd,
and is a pushforward of an Azumaya algebra from the complement~$\widetilde\P(U) \setminus \Delta$
of the ramification divisor of the double covering if~$n$ is even, see~\cite[\S\S3.5--3.6]{K08}.
So, in a sense, the classical discriminant data of a family of quadrics
(the divisor~$\Delta$ and the appropriate double covering)
are encoded in the structure of the sheaf~$\cB_0(V)$ of even Clifford algebras.

Similarly, in the case where~$k > n/2$ (so that~$X_U$ is a variety of general type),
\cite[Theorem~5.5]{K08} provided the ``opposite'' semiorthogonal decomposition
\begin{equation}
\label{eq:dbpw-cl}
\Db(\P(U), \cB_0(V)) = \Big\langle \cB_{n-2k+1}(V), \dots, \cB_{-1}(V), \cB_0(V), \Db(X_U) \Big\rangle,
\end{equation}
where~$\cB_i(V)$ are sheaves of~$\cB_0(V)$-modules that form an exceptional sequence of length~$2k - n$.

\subsection{Clifford spaces}

In this paper we discuss the extreme cases of intersections of quadrics satisfying~\eqref{eq:dbpw-cl},
namely those, where~$k \ge n$, so that~$X_U = \varnothing$.
The corresponding category~$\Db(\P(U), \cB_0(V))$, in a sense,
describes the geometry of the empty intersection of quadrics,

By~\eqref{eq:dbpw-cl}, the category~$\Db(\P(U), \cB_0(V))$ has a full exceptional collection of length~$2k - n$:
\begin{equation}
\label{eq:intro-db-clifford}
\Db(\P(U), \cB_0(V)) = \Big\langle \cB_{1-2k+n}(V), \dots, \cB_{-1}(V), \cB_0(V) \Big\rangle.
\end{equation}
It looks visually similar to the Beilinson's exceptional collection of a projective space.
The analogy becomes even more appealing, if we take into account
that there is an autoequivalence of~$\Db(\P(U), \cB_0(V))$ that takes every~$\cB_i(V)$ to~$\cB_{i+1}(V)$,
and as we will see below, the analogy goes even further.

Recall that the standard exceptional collection of a projective space~$\P^{m-1}$
extends to an infinite sequence of exceptional objects (known as a {\sf helix})~$\{ \cO_{\P^{m-1}}(i),\ i \in \ZZ \}$
the defining property of which is that
each subsequence~$\cO_{\P^{m-1}}(i)$, $\cO_{\P^{m-1}}(i+1)$, \dots, $\cO_{\P^{m-1}}(i+m-1)$ of length~$m$
is a full exceptional collection.
This helix is periodic (in an obvious sense) and strong, and its graded algebra
\begin{equation*}
\bigoplus_{i=0}^\infty \Ext^\bullet(\cO_{\P^{m-1}}, \cO_{\P^{m-1}}(i)) \cong
\bigoplus_{i=0}^\infty \Hom(\cO_{\P^{m-1}}, \cO_{\P^{m-1}}(i))
\end{equation*}
is isomorphic to the graded coordinate algebra of the projective space.
The exceptional collection~\eqref{eq:intro-db-clifford} behaves similarly:
it extends to a periodic strong helix~$\{ \cB_i(V),\ i \in \ZZ \}$
(see Theorem~\ref{thm:clifford-helix})
and the corresponding graded algebra
\begin{equation}
\label{eq:algebra-clms}
\bB_U \coloneqq \bigoplus_{i=0}^\infty \Ext^\bullet_{\cB_0(V)}(\cB_0(V), \cB_i(V)) \cong
\bigoplus_{i=0}^\infty \Hom_{\cB_0(V)}(\cB_0(V), \cB_i(V))
\end{equation}
should be thought of as its coordinate algebra.
Indeed, an analog of Serre's theorem (Theorem~\ref{thm:clifford-algebra}) holds in this case,
giving an equivalence of abelian categories
\begin{equation*}
\coh(\P(U), \cB_0(V)) \simeq \qgr(\bB_U),
\end{equation*}
where the left side is the category of coherent sheaves of~$\cB_0(V)$-modules on~$\P(U)$
and the right side is the quotient of the category of finitely generated graded~$\bB_U$-modules
by finite-dimensional modules.

Thus, $(\P(U), \cB_0(V))$ can be thought of as a noncommutative variety; we call it {\sf a Clifford space}.
In Proposition~\ref{prop:as-regularity} we show that the coordinate algebra~$\bB_U$ of any Clifford space
satisfies an important homological property:
it is \emph{Artin--Schelter regular} of dimension~$k = \dim(U)$.

\subsection{The maximal Clifford space}

The Clifford space that corresponds to the space of \emph{all quadrics},
i.e., $U = \Sym^2V^\vee$, is called {\sf the maximal Clifford space}.
From the homological projective duality point of view (see~\cite{K07} and~\cite{K14}), 
the maximal Clifford space is homologically projectively dual
to the second Veronese embedding of the projective space~$\P(V)$.

Since~$k = \dim(\Sym^2V^\vee) = n(n+1)/2$, the length of the collection~\eqref{eq:intro-db-clifford} is~$2k - n = n(n+1) - n = n^2$,
and the collection itself (we abbreviate~$\cB_i(V)$ to simply~$\cB_i$) looks as follows:
\begin{equation*}
\label{eq:db-clms}
\Db(\P(\Sym^2V^\vee), \cB_0) = \Big\langle \cB_{1-n^2}, \dots, \cB_{-1}, \cB_0 \Big\rangle.
\end{equation*}
We write~$\bB = \bB_{\Sym^2(V^\vee)}$ for the coordinate algebra of the maximal Clifford space.

To give a description of the algebra~$\bB$ note that the group~$\GL(V)$
acts on~$\P(\Sym^2V^\vee)$ and all the sheaves~$\cB_i$ are $\GL(V)$-equivariant;
therefore~$\GL(V)$ acts on the algebra~$\bB$.
Recall also that the irreducible polynomial representations of~$\GL(V)$ are indexed
by the Young diagrams with at most~$n$ rows; we denote by~$\YD_n$ the set of such Young diagrams,
by~$\Sigma^\alpha V$ the irreducible representation of~$\GL(V)$ corresponding to~$\alpha \in \YD_n$,
and by~$|\alpha|$ the number of boxes in~$\alpha$.

\begin{theorem}
\label{thm:intro-bb}
The coordinate algebra~$\bB$ of the maximal Clifford space
is the multiplicity free direct sum of all irreducible polynomial representations of~$\GL(V)$.
More precisely,
\begin{equation*}
\bB = \bigoplus_{i=0}^\infty \bB_i \,\cong\,
\bigoplus_{i=0}^\infty \left( \bigoplus_{\alpha \in \YD_n,\ |\alpha| = i} \Sigma^\alpha V \right),
\end{equation*}
where~$\bB_i$ is the $i$-th graded component of~$\bB$.
\end{theorem}

It is interesting to compare this description with the structure of the tensor algebra of~$V$:
\begin{equation*}
\bT(V) = 
\bigoplus_{i=0}^\infty V^{\otimes i} \,\cong\,
\bigoplus_{i=0}^\infty \left( \bigoplus_{\alpha \in \YD_n,\ |\alpha| = i} \rR_\alpha \otimes \Sigma^\alpha V \right),
\end{equation*}
where~$\rR_\alpha$ is the irreducible representation of the symmetric group~$\fS_i$ corresponding to~$\alpha$.

To describe the multiplication in~$\bB$
denote by~$\ell_\diag(\alpha)$ the length of the main diagonal of~$\alpha$, i.e.,
\begin{equation*}
\ell_\diag(\alpha) = \max\{i \mid \alpha_i \ge i \}.
\end{equation*}
Recall that a Young diagram~$\alpha$ is {\sf symmetric} 
if~$\alpha^T = \alpha$, where~$\alpha^T$ is the transposition of~$\alpha$;
the set of all symmetric Young diagrams with at most~$n$ rows and columns is denoted by~$\SYD_n$.
Note that for any~$\alpha \in \SYD_n$ the sum~$|\alpha| + \ell_\diag(\alpha)$ is even.

\begin{theorem}
\label{thm:intro-syzygies}
There is a $\GL(V)$-equivariant exact sequence of graded $\bB$-modules
\begin{equation}
\label{eq:simple-resolution}
0 \to 
\bF_{n(n+1)/2} \xrightarrow{\quad}
\bF_{n(n+1)/2-1} \xrightarrow{\quad}
\dots \xrightarrow{\ \rd_2\ } 
\bF_1 \xrightarrow{\ \rd_1\ } 
\bF_0 \to 
\kk \to 0,
\end{equation}
where~$\kk \coloneqq \bB/\bB_{\ge 1}$ is the simple $\bB$-module and~$\bF_i$ are free graded $\bB$-modules defined by
\begin{equation}
\label{eq:terms}
\bF_i \coloneqq \bigoplus_{\substack{\alpha \in \SYD_n \\[.5ex]|\alpha| + \ell_{\diag}(\alpha) = 2i}} 
\Sigma^\alpha V \otimes \bB(- |\alpha|)
\end{equation}
where~$\bB(-k)$ is the free graded $\bB$-module with the grading shifted by~$k$.
\end{theorem}

The smallest symmetric Young diagrams are 
\begin{itemize}
\ytableausetup{boxsize = .4em, centertableaux}
\item
the empty diagram~$\varnothing$; note that~$|\varnothing| = \ell_\diag(\varnothing) = 0$;
\item
the one-box diagram~$\ydiagram{1} = (1)$; note that~$|\ydiagram{1}| = \ell_\diag(\ydiagram{1}) = 1$;
\item
the small hook diagram~$\ydiagram{2,1} = (2,1)$; note that~$\left|\ydiagram{2,1}\right| = 3$, $\ell_\diag(\ydiagram{2,1}) = 1$.
\end{itemize}
Therefore, the rightmost part of the exact sequence~\eqref{eq:simple-resolution} looks like
\begin{equation}
\label{eq:simple-resolution-explicit}
\dots \to
\Sigma^{2,1}V \otimes \bB(-3) \xrightarrow{\ \rd_2\ } V \otimes \bB(-1) \xrightarrow{\ \rd_1\ } \bB \to \kk \to 0.
\end{equation}
It is easy to see that:
\begin{itemize}
\item
Exactness of~\eqref{eq:simple-resolution-explicit} in term~$\bB$ means that the algebra~$\bB$
is generated by its first component, i.e., the morphism (induced by the multiplication in~$\bB$)
\begin{equation*}
\bT(V) = \bT(\bB_1) \to \bB
\end{equation*}
from the tensor algebra of~$V = \bB_1$ is surjective, see Corollary~\ref{cor:intro-generation}.
\item
Exactness of~\eqref{eq:simple-resolution-explicit} in term~$V \otimes \bB(-1)$ means
that the kernel of the epimorphism~$\bT(V) \to \bB$ is the two-sided ideal in~$\bT(V)$
generated by the subspace~$\Sigma^{2,1}V \subset V^{\otimes 3}$,
see Corollary~\ref{cor:intro-relations}.
\end{itemize}
More generally, it follows from Theorem~\ref{thm:intro-syzygies}
that the dual $\rA_\infty$-algebra~$\bB^! \coloneqq \Ext^\bullet_\bB(\kk,\kk)$ of~$\bB$
can be written as
\begin{equation*}
\bB^! \cong
\bigoplus_{\alpha \in \SYD_n} \Sigma^\alpha V^\vee \left(|\alpha|\left)\left[-\tfrac{|\alpha| + \ell_{\diag}(\alpha)}2\right]\right.\right.,
\end{equation*}
where~$(-)$ and~$[-]$ are the shifts of the internal and homological grading, respectively,
see Corollary~\ref{cor:intro-bb-shriek}.

\subsection{Minimal Clifford spaces}

Since intersection of less than~$n$ quadrics in~$\P(V) = \P^{n-1}$ is always non-empty,
minimal Clifford spaces have the form~$(\P(U), \cB_0(V))$,
where~$U \subset \Sym^2V^\vee$ is a subspace of dimension~$\dim(U) = n$
such that the intersection of quadrics~$Q_u$ for~$u \in U$ is empty.

The full exceptional collection of a minimal Clifford space has length~$2n - n = n$, so that
\begin{equation*}
\Db(\P(U), \cB_0(V)) = \Big\langle \cB_{1-n}(V), \dots, \cB_{-1}(V), \cB_0(V) \Big\rangle.
\end{equation*}
We check that~$(\P(U), \cB_0(V))$ is an example of a \emph{noncommutative projective space},
namely its exceptional collection extends to a periodic helix of the form~$\Big\{ \cB_i(V), i \in \ZZ \Big\}$,
this helix is geometric in the sense of Bondal--Polishchuk~\cite{BP93},
its graded algebra~$\bB_U$ (defined in~\eqref{eq:algebra-clms}) is Koszul,
and its dual algebra~$\bB_U^!$ is Frobenius, see Theorem~\ref{thm:bbu}.
Moreover, $\bB_U$ is a flat deformation of the polynomial algebra,
while~$\bB_U^!$ is a flat deformation of the exterior algebra, see Corollary~\ref{cor:bbu-generation}.

\subsection*{Conventions}

We work over a field~$\kk$ of characteristic zero.

\subsection*{Acknowledgement}

I would like to thank Pieter Belmans, Anton Fonarev and Dmitri Orlov for useful discussions.
This work is carried out with a support of Russian Science Foundation, grant~25-11-00057, \url{https://rscf.ru/project/25-11-00057/}.

\section{Clifford spaces}
\label{sec:cs}

Let~$V$ be a vector space of dimension~$n$
and let $\P(\Sym^2(V^\vee))$ be the projective space of quadratic forms on~$V$.
Let~$U \subset \Sym^2V^\vee$ be a linear subspace of dimension~$k$.
For~$u \in U$ we denote by~$Q_u \subset \P(V)$ the corresponding quadric.
We assume that the intersection of the quadrics in~$U$ is empty:
\begin{equation}
\label{eq:empty}
\bigcap_{u \in U} Q_u = \varnothing.
\end{equation}
In particular, $k \ge n$.

\subsection{Clifford sheaves}

Following~\cite{K08} we denote by~$\cB_0(V)$ and~$\cB_1(U)$ the sheaves of even and odd parts of Clifford algebras on~$\P(U)$
corresponding to the family of quadrics~$\{ Q_u \}_{u \in \P(U)}$ defined as
\begin{align}
\label{eq:cliff-0}
\cB_0(V) &\coloneqq \cO \oplus \big(\bw2V \otimes \cO(-1)\big) \oplus \big(\bw4V \otimes \cO(-2)\big) \oplus \dots,
\\
\label{eq:cliff-1}
\cB_1(V) &\coloneqq \big(V \otimes \cO\big) \oplus \big(\bw3V \otimes \cO(-1)\big) \oplus \big(\bw5V \otimes \cO(-2)\big) \oplus \dots,
\end{align} 
respectively. 
We consider the Clifford multiplication
\begin{equation*}
\bw{l}V \otimes \bw{m} V \otimes \cO \to \bigoplus_{s = 0}^{\min(l,m)} \bw{l+m-2s}V \otimes \cO(s),
\end{equation*}
defined by
\begin{multline}
\label{eq:clifford-multiplication}
(v_1 \wedge \dots \wedge v_l) \otimes (v'_1 \wedge \dots \wedge v'_m)
\\
\mapsto
\sum_{s = 0}^{\min(l,m)}
\sum_{\sigma \in \fS_l,\ \tau \in \fS_m}
\pm \left(\bigwedge_{i = s+1}^l v_{\sigma(i)}\right) \wedge
\left(\bigwedge_{i = s+1}^m v'_{\tau(i)} \right) \otimes
\left(\bigotimes_{i=1}^s \bq_U(v_{\sigma(i)}, v'_{\tau(i)})\right)
\end{multline}
where~$\bq_U$ in the right-hand side is the composition~$\Sym^2(V) \otimes \cO \to U^\vee \otimes \cO \to \cO(1)$
of the dual map to the embedding~$U \hookrightarrow \Sym^2(V^\vee)$ and the tautological epimorphism of~$\P(U)$,
and the signs are chosen using the Koszul rule of signs (see~\cite{K14} for details).
The Clifford multiplication endows~$\cB_0(V)$ with an $\cO$-algebra structure.
We consider the category~$\coh(\P(U), \cB_0(V))$ of sheaves of~$\cB_0(V)$-modules on~$\P(U)$ as a noncommutative variety
and call it {\sf a Clifford space} associated to the space of quadrics~$U$.

The Clifford multiplication endows~$\cB_1(V)$ with a structure of~$\cB_0(V)$-bimodule.
Furthermore, using twists, we obtain and infinite sequence of sheaves of $\cB_0(V)$-bimodules
\begin{equation}
\label{eq:cliff-i}
\cB_{i+2}(V) = \cB_i(V) \otimes \cO(1).
\end{equation}
Combining this with~\eqref{eq:cliff-0} and~\eqref{eq:cliff-1}, we can uniformly rewrite their definition as
\begin{equation}
\label{eq:cliff-i-again}
\cB_i(V) =
\bigoplus_{s \in \ZZ} \bw{i-2s}V \otimes \cO(s),
\end{equation}
where by convention~$\bw{j}V = 0$ unless~$0 \le j \le n$.
This easily implies the following

\begin{lemma}
\label{lem:cohomology-clifford}
We have
\begin{equation*}
\rH^p(\P(U), \cB_{i}(V)) \cong
\begin{cases}
\bigoplus_{s \ge 0} \bw{i-2s}V \otimes \Sym^sU^\vee,
& \text{if~$p = 0$ and~$i \ge 0$},
\\
\bigoplus_{s \ge 0} \bw{i + 2k + 2s}V \otimes \Sym^sU,
& \text{if~$p = k$ and~$i \le n-2k$},
\\
0, & \text{otherwise}.
\end{cases}
\end{equation*}
\end{lemma}

The bimodules~$\cB_i(V)$ behave as invertible sheaves.

\begin{lemma}
\label{lem:clifford-twist}
The sheaves~$\cB_i(V)$ are locally projective over~$\cB_0(V)$ as left or right modules and
\begin{equation}
\label{eq:ext-clifford}
\Ext_{\cB_0(V)}^\bullet(\cB_{i}(V), \cB_{j}(V)) \cong
\rH^\bullet(\P(U), \cB_{j-i}(V)).
\end{equation}
Moreover, the functors
\begin{equation}
\label{eq:twist}
\cF \mapsto \cF(\tfrac{i}2) \coloneqq \cF \otimes_{\cB_0(V)} \cB_i(V) \cong \cHom_{\cB_0(V)}(\cB_{-i}(V), \cF)
\end{equation}
are exact autoequivalences of~$\coh(\P(U),\cB_0(V))$ such that~$\cF(\tfrac{i}2)(\tfrac{j}2) \cong \cF(\tfrac{i+j}2)$.
\end{lemma}

\begin{proof}
For local projectivity of~$\cB_i(V)$ and for the following useful isomorphisms
\begin{equation}
\label{eq:tensor-bi-bj}
\cB_i(V) \otimes_{\cB_0(V)} \cB_j(V) \cong \cB_{i+j}(V),
\qquad
\cRHom_{\cB_0(V)}(\cB_i(V),\cB_j(V)) \cong \cB_{j-i}(V)
\end{equation}
for all~$i,j \in \ZZ$ (where the tensor product is derived) we refer to~\cite[Lemma~3.8 and Corollary~3.9]{K08}.
The isomorphism~\eqref{eq:ext-clifford} and the last claim follow easily.
\end{proof}

We refer to the autoequivalence~\eqref{eq:twist} as {\sf Clifford twist}.
Since~$\cB_{2i}(V) \cong \cB_0(V) \otimes \cO(i)$ by~\eqref{eq:cliff-i},
the integral Clifford twist~$\cF(i)$ is isomorphic to~$\cF \otimes \cO(i)$,
while~$\cF(\tfrac{i}2)$ is its square root; this justifies the notation.

\subsection{The helix}

Let~$(\cE_1,\dots,\cE_m)$ be a full exceptional collection in a triangulated category~$\cT$.
Recall that a {\sf helix} in~$\cT$ extending~$(\cE_1,\dots,\cE_m)$
is an infinite sequence~$\{\cE_i\}_{i \in \ZZ}$ of exceptional objects such that~$(\cE_{i},\dots,\cE_{i+m-1})$
is a full exceptional collection in~$\cT$ for all~$i$
(this definition differs slightly from the one used in~\cite{BP93}, see Remark~\ref{rem:helix}).
We say that a helix is
{\sf $1$-periodic} if there is an autoequivalence~$\Phi \colon \cT \to \cT$ such that~$\cE_{i+1} = \Phi(\cE_i)$ for all~$i$,
and {\sf strong} if~$\Ext^p(\cE_i,\cE_j) = 0$ for all~$i \le j$ and~$p \ne 0$.

The following result proved in~\cite{K08} and already mentioned in the introduction, 
can be considered as a ``Clifford version'' of the Beilinson's exceptional collection on a projective space.

\begin{theorem}
\label{thm:clifford-helix}
Let~$U \subset \Sym^2(V^\vee)$ be a $k$-dimensional space of quadrics satisfying the assumption~\eqref{eq:empty}.
The derived category~$\Db(\P(U), \cB_0(V))$ has a full and strong exceptional collection
\begin{equation}
\label{eq:db-clifford}
\Db(\P(U), \cB_0(V)) = \Big\langle \cB_{1+ n - 2k}(V), \dots, \cB_{-1}(V), \cB_0(V) \Big\rangle.
\end{equation}
The sequence~$\{\cB_i(V))\}_{i \in \ZZ}$ is a $1$-periodic strong helix extending the exceptional collection~\eqref{eq:db-clifford}.
\end{theorem}

\begin{proof}
The fact that~\eqref{eq:db-clifford} is a full exceptional collection follows from~\cite[Theorem~5.5]{K08}
applied to the family of quadrics~\mbox{$U \subset \Sym^2(V^\vee)$}
(note that the space~$V$ is denoted by~$W$ in~\cite{K08} and~$U$ is denoted by~$L$), and twisted by~$\cB_1(V)$.
To check that the collection is strong we apply~\eqref{eq:ext-clifford} and Lemma~\ref{lem:cohomology-clifford}.

Next, applying the Clifford twist by~$i/2$ to~\eqref{eq:db-clifford} and using~\eqref{eq:tensor-bi-bj}, we see that
\begin{equation*}
\Db(\P(U)), \cB_0(V)) = \langle \cB_{i}(V), \cB_{i+1}(V), \dots, \cB_{i + 2k - n - 1}(V) \rangle
\end{equation*}
is a full exceptional collection;
this shows that~$\{\cB_i(V))\}_{i \in \ZZ}$ is a helix.
This helix is 1-periodic, because the Clifford twist by~$1/2$ shifts it by~$1$,
and strong by~\eqref{eq:ext-clifford} and Lemma~\ref{lem:cohomology-clifford}.
\end{proof}

A helix~$\{\cE_i\}$ gives rise to a $\ZZ$-algebra $\bigoplus_{i \le j} \Hom(\cE_i,\cE_j)$ (see~\cite{BP93}),
and if a helix is 1-periodic, it gives rise to a graded algebra~$\bigoplus_{i \ge 0} \Hom(\cE_0, \cE_i)$.
In the case of a Clifford space, the corresponding algebra can be thought of as the coordinate algebra of the Clifford space;
this is justified by the following result.

\begin{theorem}
\label{thm:clifford-algebra}
Let~$U \subset \Sym^2(V^\vee)$ be a $k$-dimensional space of quadrics satisfying the assumption~\eqref{eq:empty}.
The graded algebra of the~$1$-periodic helix~$\{\cB_i(V))\}_{i \in \ZZ}$ in~$\Db(\P(U), \cB_0(V))$
\begin{equation*}
\bB_U := \bigoplus_{i=0}^\infty \Hom_{\cB_0(V)}(\cB_0(V), \cB_i(V))
\end{equation*}
is right noetherian and there is an equivalence of categories
\begin{equation*}
\coh(\P(U), \cB_0(V)) \simeq \qgr(\bB_U),
\end{equation*}
where~$\qgr(\bB_U)$ is the quotient of the category of finitely generated graded~$\bB_U$-modules
by the subcategory of finite-dimensional graded~$\bB_U$-modules.
\end{theorem}

\begin{proof}
This follows easily from the theorem of Artin and Zhang~\cite[Theorem~4.5]{AZ}
applied to the abelian category~$\coh(\P(U), \cB_0(V))$,
the object~$\cB_0(V)$ in it, considered as the structure sheaf,
and the autoequivalence~$\cF \mapsto \cF(\tfrac12)$ defined by~\eqref{eq:twist}.
Indeed, it is not hard to see that the hypothesis of the theorem are satisfied:
the hypothesis~(H1) holds because~$\cB_0(V)$ is a coherent algebra on the projective scheme~$\P(U)$,
(H2) is obvious, because~$\bB_{U,0} = \rH^0(\P(U),\cB_0(V)) = \kk$,
and~(H3) is easy, because the square of the autoequivalence~$\cF \mapsto \cF(\tfrac12)$
is the twist by an ample line bundle on~$\P(U)$.
Thus, \cite[Theorem~4.5]{AZ} proves that~$\bB_U$ is right noetherian
and~$\coh(\P(U), \cB_0(V)) \simeq \qgr(\bB_U)$.
\end{proof}

\begin{remark}
The algebra~$\bB_U$ contains a large central subalgebra
\begin{equation*}
\Sym^\bullet(U^\vee) =
\bigoplus_{j = 0}^\infty \, \Sym^jU^\vee =
\bigoplus_{j = 0}^\infty \, \rH^0(\P(U), \cO_{\P(U)}(j)) \subset
\bigoplus_{j = 0}^\infty \, \rH^0(\P(U), \cB_{2j}(V)) =
\bB_U^{(2)} \subset \bB_U,
\end{equation*}
where~$\bB_U^{(2)}$ denotes the second Veronese subalgebra in~$\bB_U$
and the embedding in the middle is induced by the inclusion~$\cO \subset \cB_0(V)$.
Moreover, the algebra~$\bB_U$ is finitely generated as~$\Sym^\bullet(U^\vee)$-module.
Actually, the center of~$\bB_U$ is slightly larger than~$\Sym^\bullet(U^\vee)$.
If~$n$ is even, it is generated over~$\Sym^\bullet(U^\vee)$ by the subspace~$\bw{n} V \subset \rH^0(\P(U), \cB_n(V))$;
it gives the double covering~$\widetilde\P(U) \to \P(U)$ branched along the discriminant locus~$\Delta(U) \subset \P(U)$
and the Clifford sheaves~$\cB_i(V)$ can be realized as pushforwards of appropriate coherent sheaves from~$\widetilde\P(U)$;
similarly, if~$n$ is odd, the sheaves~$\cB_i(V)$ can be realized as pushforwards from the root stack~$\sqrt{\P(U),\Delta(U)}$,
see~\cite[\S\S3.5--3.6]{K08}.
Note that~$\Delta(U)$ is the intersection of~$\P(U) \subset \P(\Sym^2V^\vee)$
with the discriminant divisor~$\Delta \subset \P(\Sym^2V^\vee)$,
and if~$U$ is general, $\Sing(\Delta(U)) = \P(U) \cap \Delta^{\ge 2}$,
where~$\Delta^{\ge 2} \subset \P(\Sym^2V^\vee)$ is the corank~$2$ degeneracy locus.
\end{remark}

\begin{example}
Let~$n = 3$, $k = 4$, and let~$U \subset \Sym^2V^\vee$ be general.
Then~$\Delta(U) \subset \P(U) = \P^3$ is the Cayley cubic surface with~$\Sing(\Delta(U))$ a 4-point set.
The category~$\Db(\P(U),\cB_0(V))$ can be thought of as a twisted categorical resolution of the root stack~$\sqrt{\P(U),\Delta(U)}$.
It has an exceptional collection of length~$2k - n = 5$, which is remarkably small.
\end{example}

\section{Maximal Clifford space}
\label{sec:cms}

Recall from the introduction that the {\sf maximal Clifford space} is the Clifford space
\begin{equation*}
(\P(\Sym^2(V^\vee)), \cB_0(V)),
\end{equation*}
corresponding to the complete linear system of all quadrics in~$\P(V)$.

Throughout this section we use abbreviation~$\cB_i$ for~$\cB_i(V)$.
Note that the algebra~$\cB_0$ is $\PGL(V)$-equivariant, hence the group~$\PGL(V)$ acts on the maximal Clifford space.
On the other hand, the sheaves~$\cB_i$ for~$i \ne 0$ are not $\PGL(V)$-equivariant
(but, say, tensor products $\Sym^iV \otimes \cB_{-i}$ are).

Theorems~\ref{thm:clifford-helix} and~\ref{thm:clifford-algebra} apply to the maximal Clifford space
and give the full strong exceptional collection
\begin{equation}
\label{eq:db-max-cs}
\Db(\P(\Sym^2(V^\vee)), \cB_0) = \Big\langle \cB_{1-n^2}, \dots, \cB_{-1}, \cB_0 \Big\rangle,
\end{equation}
the 1-periodic strong helix~$\{ \cB_i \}_{i \in \ZZ}$, and the graded algebra~$\bB_{\Sym^2(V^\vee)}$
(abbreviated to~$\bB$), such that
\begin{equation*}
\coh(\P(\Sym^2(V^\vee)), \cB_0) \simeq \qgr(\bB).
\end{equation*}
In this section we discuss some specific properties of the maximal Clifford spaces.
We start with an example of a maximal Clifford space of small dimension.

\begin{example}
\label{ex:cls-2}
Consider the Clifford space with~$\dim(V) = 2$.
Then~$\P(\Sym^2V) \cong \P^2$ and the sheaf of algebras~$\cB_0 \cong \cO \oplus \cO(-1)$ is commutative,
hence the category~$\Db(\P(\Sym^2(V^\vee)), \cB_0)$ is equivalent to the derived category of the relative spectrum
\begin{equation*}
\Spec_{\P(\Sym^2V^\vee)}(\cB_0) =
\Spec_{\P(\Sym^2V^\vee)}(\cO \oplus \cO(-1)) \cong
\P(V^\vee) \times \P(V^\vee),
\end{equation*}
where the natural morphism~$\P(V^\vee) \times \P(V^\vee) \to \P(\Sym^2V^\vee)$ is given by~$(v_1,v_2) \mapsto v_1 \cdot v_2$;
it is the double covering ramified over the Veronese conic~$\P(V^\vee) \subset \P(\Sym^2V^\vee)$.
Moreover, under this equivalence the Clifford twist~$\cF \mapsto \cF(\tfrac12)$ (see Lemma~\ref{lem:clifford-twist})
corresponds to the composition of the twist~\mbox{$\cF \mapsto \cF(1,0)$}
with the involution~$(v_1,v_2) \mapsto (v_2,v_1)$,
and therefore the Clifford sheaves of~$\cB_0$-modules~$\cB_{2i}$ and~$\cB_{2i+1}$ on~$\P(\Sym^2V^\vee)$
correspond to the line bundles~$\cO_{\P(V^\vee) \times \P(V^\vee)}(i,i)$ and~$\cO_{\P(V^\vee) \times \P(V^\vee)}(i+1,i)$, respectively.
Finally, the exceptional collection~\eqref{eq:db-max-cs} corresponds to
\begin{equation*}
\Db(\P(V^\vee) \times \P(V^\vee)) =
\langle
\cO_{\P(V^\vee) \times \P(V^\vee)}(-1,-2),
\cO_{\P(V^\vee) \times \P(V^\vee)}(-1,-1),
\cO_{\P(V^\vee) \times \P(V^\vee)}(0,-1),
\cO_{\P(V^\vee) \times \P(V^\vee)}
\rangle,
\end{equation*}
a standard exceptional collection on~$\P(V^\vee) \times \P(V^\vee)$
and the algebra~$\bB$ can be rewritten as
\begin{equation*}
\bB =
\bigoplus_{i=0}^\infty \bB_i =
\bigoplus_{i=0}^\infty \rH^0(\P(V) \times \P(V), \cO(\lceil\tfrac{i}2\rceil, \lfloor\tfrac{i}2\rfloor)) \cong
\bigoplus_{i=0}^\infty \Sym^{\lceil\tfrac{i}2\rceil}(V^\vee) \otimes \Sym^{\lfloor\tfrac{i}2\rfloor}(V^\vee),
\end{equation*}
where the multiplication~$\bB_i \otimes \bB_j \to \bB_{i+j}$ is obvious when either~$i$ or~$j$ is even,
and in the case where both~$i$ and~$j$ are odd it is induced by the identity map
\begin{equation*}
\bB_1 \otimes \bB_1 = (V^\vee \otimes \kk) \otimes (V^\vee \otimes \kk) \to V^\vee \otimes V^\vee = \bB_2.
\end{equation*}
Note, in particular, that the algebra~$\bB$ has no quadratic relations.
\end{example}

\begin{remark}
In fact, the algebra~$\bB$ discussed in Example~\ref{ex:cls-2}
is the twisted homogeneous coordinate ring (in the sense of~\cite[\S3]{SVDB01}) of~$\P(V^\vee) \times \P(V^\vee)$,
and its $\ZZ$-algebra version was discussed in~\cite[\S4.1]{VDB11}.
\end{remark}

\subsection{The algebra~$\bB$}

To prove Theorem~\ref{thm:intro-bb}, we need to recall some notation;
in particular, $\YD_n$ denotes the set of all Young diagrams with at most~$n$ rows
and for~$\alpha \in \YD_n$ we write~$|\alpha| = \sum \alpha_i$ for the number of boxes in~$\alpha$
and~$\Sigma^\alpha V$ for the Schur functor of~$\alpha$ applied to~$V$ (see~\cite[\S2.1]{W}).
Note that~$\Sigma^\alpha V$ is a polynomial representation of~$\GL(V)$.

\begin{proof}[Proof of Theorem~\textup{\ref{thm:intro-bb}}]
First of all, combining\eqref{eq:ext-clifford} and Lemma~\ref{lem:cohomology-clifford} we see that
\begin{equation*}
\bB_i \coloneqq \Ext^\bullet_{\cB_0}(\cB_0,\cB_i) \cong \bigoplus_{s = 0}^{\lfloor i/2 \rfloor} \bw{i-2s}V \otimes \Sym^s(\Sym^2V).
\end{equation*}
To check that this is the direct sum of all~$\Sigma^\alpha V$,
recall that (see, e.g., \cite[Proposition~2.3.8]{W})
\begin{equation*}
\Sym^s(\Sym^2V) = \bigoplus_{|\beta| = s} \Sigma^{2\beta} V,
\end{equation*}
where for a Young diagram $\beta = (\beta_1,\beta_2,\dots,\beta_n)$ we define~$2\beta := (2\beta_1,2\beta_2,\dots,2\beta_n)$.
Therefore, by Pieri's rule~\cite[Corollary~2.3.5]{W} we have
\begin{equation*}
\bw{i-2s}V \otimes \Sym^{s}(\Sym^2(V)) = \bigoplus_{|\beta| = s} \left( \bigoplus_{\alpha \in \rE(2\beta,i-2s)} \Sigma^{\alpha} V \right),
\end{equation*}
where $\rE(\gamma,p)$ is the set of all Young diagrams $\alpha$ such that
\begin{equation*}
\gamma_t \le \alpha_t \le \gamma_t + 1
\quad\text{for all $1 \le t \le n$}
\qquad\text{and}\qquad
|\alpha| = |\gamma| + p.
\end{equation*}
It remains to note that for each Young diagram $\alpha$ there exist a unique Young diagram~$\beta$ and integer~$p$ 
such that $\alpha \in \rE(2\beta,p)$;
indeed, if $\alpha = (\alpha_1,\alpha_2,\dots,\alpha_n)$ then
\begin{equation*}
\beta_t = \lfloor \alpha_t/2 \rfloor
\qquad\text{and}\qquad 
p = |\alpha| - 2|\beta|.
\end{equation*}
Thus, every Schur functor $\Sigma^\alpha V$ appears in the decomposition of the summand $\bw{p}V \otimes \Sym^{|\beta|}(\Sym^2(V))$
of~$\rH^0(\P(\Sym^2(V^\vee)), \cB_{|\alpha|})$, and exactly once.
\end{proof}

It is not hard to compute the multiplication map of the graded algebra~$\bB$ in small degrees.

\begin{example}
\label{ex:bb2}
First of all, using the Clifford multiplication formula~\eqref{eq:clifford-multiplication}
it is easy to see that the map
\begin{equation*}
V \otimes V = \bB_1 \otimes \bB_1 \to \bB_2 = \bw2V \oplus \Sym^2V
\end{equation*}
is given by~$v_1 \otimes v_2 \mapsto v_1 \wedge v_2 + v_1 \cdot v_2$.
Therefore, it is an isomorphism.
This shows that the algebra~$\bB$ has no quadratic relations.
\end{example}

\begin{example}
\label{ex:bb3}
Now consider the multiplication map
\begin{equation*}
V \otimes V \otimes V = \bB_1 \otimes \bB_1 \otimes \bB_1 \to \bB_3 = \bw3V \oplus \Big(V \otimes \Sym^2V\Big).
\end{equation*}
This time~\eqref{eq:clifford-multiplication} shows that this map is given by the formula
\begin{equation*}
v_1 \otimes v_2 \otimes v_3 \mapsto 
v_1 \wedge v_2 \wedge v_3 +
\Big(v_1 \otimes (v_2 \cdot v_3) - v_2 \otimes (v_1 \cdot v_3) + v_3 \otimes (v_1 \cdot v_2) \Big).
\end{equation*}
Its restriction to~$V \otimes \Sym^2V \subset V \otimes V \otimes V$
is an isomorphism onto the second summand in~$\bB_3$,
its restriction to~$\bw3V \subset V \otimes V \otimes V$
is an isomorphism onto the first summand in~$\bB_3$,
and its restriction to
\begin{equation}
\label{eq:br3}
\rR_3 \coloneqq
\Big\langle v_1 \otimes v_2 \otimes v_3 + v_2 \otimes v_1 \otimes v_3 - v_3 \otimes v_2 \otimes v_1 - v_3 \otimes v_1 \otimes v_2
\ \Big\vert\
v_1,v_2,v_3 \in V \Big\rangle
\subset V \otimes V \otimes V
\end{equation}
is zero.
Furthermore, it is easy to check that as a representation of~$\GL(V)$, we have~$\rR_3 \cong \Sigma^{2,1}V$,
and it follows that the map~$\bB_1 \otimes \bB_1 \otimes \bB_1 \to \bB_3$ is surjective,
and the space~$\rR_3$ is its kernel, and therefore~$\rR_3$ is the space of cubic relations of~$\bB$.
\end{example}

\begin{example}
\label{ex:bb4}
A similar computation shows that the map~$\bB_1^{\otimes 4} \to \bB_4$ is surjective and
\begin{equation*}
\Ker\Big( \bB_1^{\otimes 4} \to \bB_4 \Big) = 
\Big(\Sigma^{3,1}V\Big)^{\oplus 2} \oplus \Big(\Sigma^{2,2}V\Big) \oplus \Big(\Sigma^{2,1,1}V\Big)^{\oplus 2}.
\end{equation*}
On the other hand,
\begin{equation*}
\Big(\Sigma^{2,1}V \otimes V \Big) \oplus \Big(V \otimes \Sigma^{2,1}V \Big) \cong
\Big(\Sigma^{3,1}V\Big)^{\oplus 2} \oplus \Big(\Sigma^{2,2}V\Big)^{\oplus 2} \oplus \Big(\Sigma^{2,1,1}V\Big)^{\oplus 2}.
\end{equation*}
This computation suggests that the syzygy space of the algebra~$\bB$ in degree~$4$ is~$\Sigma^{2,2}V$.
\end{example}

In the next subsection we check that this is indeed the case and prove a generalization of this fact.
In particular, we describe the dual $\mathrm{A}_\infty$-algebra of~$\bB$.

\subsection{The Clifford--Koszul complex}
\label{ss:ckc}

In this subsection we introduce a Clifford analog of the Koszul complex.
To state the result we recall some notation.

Given a Young diagram~$\alpha$ we denote by~$\alpha^T$ the transposed Young diagram.
We also write~$\ell_{\diag}(\alpha)$ for the length of the diagonal of~$\alpha$, i.e.,
\begin{equation*}
\ell_{\diag}(\alpha) = \max\{t \mid \alpha_t \ge t\}.
\end{equation*}
We denote by~$\SYD_n \subset \YD_n$ the set of all symmetric Young diagrams, i.e.,
\begin{equation*}
\SYD_n = \{ \alpha \in \YD_n \mid \alpha^T = \alpha \}.
\end{equation*}
Symmetric diagrams satisfy the following obvious property:
\begin{equation}
\label{eq:syd-mod-len}
|\alpha| \equiv \ell_\diag(\alpha) \bmod 2
\qquad\text{if~$\alpha \in \SYD_n$.}
\end{equation} 
For symmetric Young diagrams~$\beta \subset \alpha$ such that~$|\alpha| = |\beta| + 1$
by the Littlewood--Richardson rule (see~\cite[Theorem~2.3.4]{W}) there is a canonical morphism of Schur functors
\begin{equation*}
\Sigma^\alpha V \xrightarrow{\ \theta_1\ } \Sigma^\beta  V \otimes V,
\end{equation*}
and in the case where~$\beta \subset \alpha$ and~$|\alpha| = |\beta| + 2$
there are two canonical morphisms of Schur functors
\begin{equation*}
\Sigma^\alpha V \xrightarrow{\ \theta_2\ } \Sigma^\beta  V \otimes \bw2V,
\qquad 
\Sigma^\alpha V \xrightarrow{\ \theta_2^+\ } \Sigma^\beta  V \otimes \Sym^2V,
\end{equation*}
respectively.
We also denote by
\begin{equation*}
V \otimes \cB_k \xrightarrow{\ \kappa_1\ } \cB_{k+1},
\qquad \text{and} \qquad
\bw2V \otimes \cB_k \xrightarrow{\ \kappa_2\ } \cB_{k+2},
\end{equation*}
the morphisms induced by the embeddings~$\bw{i}V \otimes \cO \hookrightarrow \cB_i$ and the Clifford multiplication.
Finally, we write~$\bq \colon \Sym^2V \otimes \cO \to \cO(1)$ for the tautological epimorphism on~$\P(\Sym^2V^\vee)$.

\begin{theorem}
\label{thm:clifford-koszul}
There is a $\PGL(V)$-equivariant exact sequence of $\cB_0$-modules on~$\P(\Sym^2V^\vee)$
\begin{equation}
\label{eq:clifford-koszul}
0 \to \cG_{n(n+1)/2} \to \cG_{n(n+1)/2-1} \to \dots \to \cG_1 \to \cG_0 \to 0,
\end{equation}
where 
\begin{equation}
\label{eq:ck-cgi}
\cG_i \coloneqq \bigoplus_{\substack{\alpha \in \SYD_n \\[.5ex]|\alpha| + \ell_{\diag}(\alpha) = 2i}} 
\Sigma^\alpha V \otimes \cB_{- |\alpha|}.
\end{equation}
All differentials in~\eqref{eq:clifford-koszul} are linear combinations of three types of maps:
\begin{align}
\label{eq:dg1}
\Sigma^\alpha V \otimes \cB_{- |\alpha|} \xrightarrow{\ \theta_1\ }
\Sigma^\beta V \otimes V \otimes \cB_{- |\alpha|} \xrightarrow{\ \kappa_1\ }
\Sigma^\beta V \otimes \cB_{- |\beta|} , 
&& \text{if~$\beta \subset \alpha$ and~$|\alpha| - |\beta| = 1$,}
\\
\label{eq:dg2}
\Sigma^\alpha V \otimes \cB_{- |\alpha|} \xrightarrow{\ \theta_2\ }
\Sigma^\beta V \otimes \bw2V \otimes \cB_{- |\alpha|} \xrightarrow{\ \kappa_2\ }
\Sigma^\beta V \otimes \cB_{- |\beta|} , 
&& \text{if~$\beta \subset \alpha$ and~$|\alpha| - |\beta| = 2$,}
\\
\label{eq:dg3}
\Sigma^\alpha V \otimes \cB_{- |\alpha|} \xrightarrow{\ \theta_2^+\ }
\Sigma^\beta V \otimes \Sym^2V \otimes \cB_{- |\alpha|} \xrightarrow{\ \bq\ }
\Sigma^\beta V \otimes \cB_{- |\beta|} , 
&& \text{if~$\beta \subset \alpha$ and~$|\alpha| - |\beta| = 2$.}
\end{align}
\end{theorem}

\begin{remark}
\label{rem:ck-differentials}
We do not know an explicit expression for the differentials in~\eqref{eq:clifford-koszul}.
In particular, we do not know if all the maps~\eqref{eq:df1}, \eqref{eq:df2}, and~\eqref{eq:df3}
appear in the differentials with nonzero coefficients.
\end{remark}

The complex~\eqref{eq:clifford-koszul} will be referred to as {\sf the Clifford--Koszul complex}.
To give an idea of how it looks like, we provide two examples.

\begin{example}
\label{ex:ck-2}
If~$n = 2$, the Clifford--Koszul complex takes the form:
\begin{equation*}
0 \to 
\Sigma^{2,2}V \otimes \cB_{-4} \to 
\Sigma^{2,1}V \otimes \cB_{-3} \to 
V \otimes \cB_{-1} \to 
\cB_0 \to 
0.
\end{equation*}
Under the equivalence~$\Db(\P(\Sym^2(V^\vee)), \cB_0) \simeq \Db(\P(V^\vee) \times \P(V^\vee))$ of Example~\ref{ex:cls-2}
it corresponds to the exact sequence
\begin{equation*}
0 \to \cO(-2,-2) \to \cO(-1,-2)^{\oplus 2} \to \cO(0,-1)^{\oplus 2} \to \cO \to 0,
\end{equation*}
a gluing of~$0 \to \cO(-2,-2) \to \cO(-1,-2)^{\oplus 2} \to \cO(0,-2) \to 0$
and~$0 \to \cO(0,-2) \to \cO(0,-1)^{\oplus 2} \to \cO \to 0$.
\end{example}

\begin{example}
\label{ex:ck-3}
If~$n = 3$, the Clifford--Koszul complex takes the form:
\begin{multline*}
0 \to \Sigma^{3,3,3}V \otimes \cB_{-9} 
\to \Sigma^{3,3,2}V \otimes \cB_{-8}
\to \Sigma^{3,2,1}V \otimes \cB_{-6}
\\
\to \Sigma^{3,1,1}V \otimes \cB_{-5} \oplus \Sigma^{2,2}V \otimes \cB_{-4}
\to \Sigma^{2,1}V \otimes \cB_{-3}
\to V \otimes \cB_{-1} 
\to \cB_0 \to 0.
\end{multline*}
\end{example}

In general, there are symmetric Young diagrams with the same number of boxes but different length of the diagonals;
therefore, the same summand~$\cB_{-i}$ may appear in different terms of~\eqref{eq:clifford-koszul}.

We postpone the proof of Theorem~\ref{thm:clifford-koszul} until Subsection~\ref{ss:proof-ck}, and now we deduce some implications.

\subsection{Implications of Clifford--Koszul complex}

First of all, using Clifford--Koszul complex we construct a free resolution of the simple $\bB$-module;
this gives a proof of Theorem~\ref{thm:intro-syzygies} from the introduction.

\begin{proof}[Proof of Theorem~\textup{\ref{thm:intro-syzygies}}]
We consider the hypercohomology spectral sequence obtained
by applying the functor
\begin{equation*}
\cG \mapsto \bigoplus_{i = 0}^\infty \rH^\bullet(\P(\Sym^2V^\vee), \cG(\tfrac{i}2))
\end{equation*}
to the exact sequence~\eqref{eq:clifford-koszul} and using Lemma~\ref{lem:cohomology-clifford} to compute the cohomology.
For~$\alpha \in \SYD_n$ we have~\mbox{$|\alpha| \le n^2$}, hence~$\rH^{n(n+1)/2}(\P(\Sym^2V^\vee), \cB_{-n^2}) = \kk$
is the only higher cohomology appearing in the spectral sequence,
while the zero cohomology terms are equal to the graded modules~$\bF_i$ as in~\eqref{eq:terms}.
Therefore, the spectral sequence converges to the exact sequence~\eqref{eq:simple-resolution}.
\end{proof}

The rightmost part of the resolution~\eqref{eq:simple-resolution} of Theorem~\ref{thm:intro-syzygies}
in degree~$i \ge 1$ looks like
\begin{equation}
\label{eq:simple-module-resolution}
\dots \to
\Big(\Sigma^{3,1,1}V \otimes \bB_{i-5}\Big) \oplus \Big(\Sigma^{2,2}V \otimes \bB_{i-4}\Big) \to
\Sigma^{2,1}V \otimes \bB_{i-3} \to
V \otimes \bB_{i-1} \to
\bB_i \to
0.
\end{equation}
It allows us to make a few important conclusions about the structure of~$\bB$.

\begin{corollary}
\label{cor:intro-generation}
The graded algebra~$\bB$ is generated by its first component;
in other words, the map 
\begin{equation*}
\bT(V) = \bT(\bB_1) \to \bB
\end{equation*}
from the tensor algebra of~$V = \bB_1$ induced by the multiplication in~$\bB$ is surjective.
\end{corollary}

\begin{proof}
Indeed, the map~$\bB_1 \otimes \bB_{i-1} = V \otimes \bB_{i-1} \to \bB_{i}$ in~\eqref{eq:simple-module-resolution}
is the multiplication map.
It is surjective for~$i \ge 1$, hence by induction so is the map~$\bB_1^{\otimes i} \to \bB_i$, as required.
\end{proof}

\begin{corollary}
\label{cor:intro-relations}
The kernel of the epimorphism~$\bT(V) \to \bB$ is the two-sided ideal in~$\bT(V)$
generated by the subspace of cubic relations~$\rR_3 \subset V^{\otimes 3} \cong \bB_1^{\otimes 3}$ defined in~\eqref{eq:br3}.
\end{corollary}

\begin{proof}
The generation of the two-sided ideal of~$\bT(V) \to \bB$ by the syzygy of the simple module is standard
(see, e.g., \cite[Proposition~1.5.2]{PP05}).
\end{proof}

Using Theorem~\ref{thm:intro-syzygies} it is also easy 
to find the dual $\rA_\infty$-algebra of~$\bB$.

\begin{corollary}
\label{cor:intro-bb-shriek}
The dual $\rA_\infty$-algebra~$\bB^! \coloneqq \Ext^\bullet_\bB(\kk,\kk)$ is isomorphic to
\begin{equation*}
\bB^! \cong
\bigoplus_{\alpha \in \SYD_n} \Sigma^\alpha V^\vee \left(|\alpha|\left)\left[-\tfrac{|\alpha| + \ell_{\diag}(\alpha)}2\right]\right.\right., 
\end{equation*}
where~$(-)$ and~$[-]$ correspond to the shifts of the internal and homological grading, respectively.
\end{corollary}

\begin{remark}
If~$\alpha$ is a Young diagram, $\ell = \ell_\diag(\alpha)$
and~$\gamma_1, \dots, \gamma_\ell$ are the hooks of~$\alpha$
(i.e., symmetric diagrams with~$\ell_\diag(\gamma_i) = 1$ and~$|\gamma_i| = 2(\alpha_i-i) + 1$)
then by the Littlewood--Richardson rule there is
a unique $\GL(V)$-equivariant morphism~$\Sigma^{\gamma_1}V \otimes \dots \otimes \Sigma^{\gamma_\ell}V \to \Sigma^{\alpha}V$.
Similarly, if~$\beta_1$ and~$\beta_2$ are the Young diagrams whose hooks are~$\gamma_1, \dots, \gamma_k$
and~$\gamma_{k+1}, \dots, \gamma_\ell$, respectively, there is a unique $\GL(V)$-equivariant
morphism~$\Sigma^{\beta_1}V \otimes \Sigma^{\beta_2}V \to \Sigma^{\alpha}V$.
We expect that the multiplication~$\bm_2$ in~$\bB^!$ is given by these morphisms.
\end{remark}

It would be interesting to describe the higher multiplications in~$\bB^!$.

\subsection{The resolution of the spinor line bundle of~$\OGr(n,2n)$}

In this subsection we recall from~\cite{FK18} a complex of vector bundles on the Grassmannian~$\Gr(n,2n)$
that will be used in Subsection~\ref{ss:proof-ck} to deduce Theorem~\ref{thm:clifford-koszul}.

Consider the space
\begin{equation*}
\DV \coloneqq V \oplus V^\vee
\end{equation*}
endowed with a symmetric bilinear form~$\bq_W \in \Sym^2(\DV^\vee)$ defined by
\begin{equation}
\label{eq:w-form}
\bq_W((v_1,f_1), (v_2,f_2)) \coloneqq f_1(v_2) + f_2(v_1),
\end{equation}
where~$v_i \in V$ and~$f_i \in V^\vee$.
Note that the subspaces~$V,V^\vee \subset \DV$ are Lagrangian.
Denote by~$\SO(\DV)$ the special orthogonal group of~$(\DV,\bq_W)$ and by~$\Spin(\DV)$ the corresponding $\Spin$-group
(the universal covering of~$\SO(\DV)$), so that~$\SO(\DV) = \Spin(\DV) / \upmu_2$.

The group~$\GL(V)$ acts on~$\DV$ by~$g \cdot (v,f) = (gv, fg^{-1})$,
preserving the form~\eqref{eq:w-form};
hence we have an embedding~$\GL(V) \subset \SO(\DV)$,
which lifts to an embedding
\begin{equation}
\label{eq:tgl}
\tGL(V) \coloneqq 
\GL(V) \times_{\SO(\DV)} \Spin(\DV) \cong
\{(t,g) \in \Gm \times \GL(V) \mid \det(g) = t^2 \} \hookrightarrow \Spin(\DV).
\end{equation}
Note that the center of~$\tGL(V)$ is the group
\begin{equation}
\label{eq:cz-tgl}
\cZ_{\tGL} \coloneqq 
\{(t,g) \in \tGL(V) \mid g = s \cdot \id \} \cong 
\{ (t,s) \in \Gm \times \Gm \mid t^2 = s^n \} \cong
\begin{cases}
\Gm, & \text{if~$n$ is odd}\\
\upmu_2 \times \Gm, & \text{if~$n$ is even}
\end{cases}
\end{equation} 
and it contains the subgroup
\begin{equation}
\label{eq:cz}
\cZ_{\Spin} \coloneqq 
\{(t,g) \in \tGL(V) \mid g = \pm \id \} \cong 
\{ (t,s) \in \Gm \times \upmu_2 \mid t^2 = s^n \} \cong
\begin{cases}
\upmu_4, & \text{if~$n$ is odd}\\
\upmu_2 \times \upmu_2, & \text{if~$n$ is even}
\end{cases}
\end{equation} 
which under the embedding~\eqref{eq:tgl} identifies with the center of~$\Spin(\DV)$.

We denote by~$(\det V)^{\frac12}$ the one-dimensional representation of the group~$\tGL(V)$
induced by the projection~$\tGL(V) \to \Gm$, $(t,g) \mapsto t$.
Note that by~\eqref{eq:tgl} its tensor square is isomorphic to~$\det(V)$.

We write~$\rcB(\DV) = \rcB_0(\DV) \oplus \rcB_1(\DV)$ for the $\ZZ/2$-graded Clifford algebra of~$\DV$.
We also denote by~$\bS_0$ and~$\bS_1$ the spinor modules over~$\rcB_0(\DV)$.
Note that~$\bS_i$ has a structure of representation of~$\Spin(\DV)$,
and the restrictions of~$\bS_i$ to~$\tGL(V)$ can be identified with the following representations:
\begin{equation}
\label{eq:bs01}
\bS_0 \cong \left(\bigoplus_{i=0}^n \bw{2i}V\right) \otimes (\det V)^{-\frac12}
\qquad\text{and}\qquad 
\bS_1 \cong \left(\bigoplus_{i=0}^{n-1} \bw{2i+1}V\right) \otimes (\det V)^{-\frac12}.
\end{equation}
We extend the definition of~$\bS_k$ for~$k \in \{0,1\}$ to all integers~$k$ by setting~$\bS_{k+2} \coloneqq \bS_k$.
Note that the right Clifford action of~$\DV = V \oplus V^\vee \subset \rcB_1(\DV)$ on~$\bS_0$ and~$\bS_1$
is given by the maps
\begin{equation*}
\bw{k}V \otimes V \to \bw{k+1}V
\qquad\text{and}\qquad 
\bw{k}V \otimes V^\vee \to \bw{k-1}V,
\end{equation*}
where the first is wedge product and the second is convolution.
Note also that the above maps (and the multiplication in the algebra) are $\Spin(\DV)$-equivariant.

Let~$\OGr_+(n,\DV) \subset \Gr(n,\DV)$ be the connected component
of the orthogonal isotropic Grassmannian containing~$[V]$,
let~$j \colon \OGr_+(n,\DV) \hookrightarrow \Gr(n,\DV)$ be the embedding morphism,
and let~$\cS$ be the spinor line bundle on~$\OGr_+(n,\DV)$, so that~$\cS \otimes \cS \cong j^*\cO_{\Gr(n,\DV)}(1)$.
Finally, let~$\cU \subset \DV \otimes \cO$ denote the tautological subbundle on~$\Gr(n,\DV)$.

Recall that for Young diagrams~$\beta \subset \alpha$ there are canonical morphisms of Schur functors
\begin{equation*}
\Sigma^\alpha\cU \to \cU \otimes \Sigma^\beta \cU,
\qquad 
\Sigma^\alpha\cU \to \bw2\cU \otimes \Sigma^\beta \cU,
\qquad 
\Sigma^\alpha\cU \to \Sym^2\cU \otimes \Sigma^\beta \cU
\end{equation*}
(induced by the maps~$\theta_1$, $\theta_2$, and~$\theta_2^+$ defined in~\S\ref{ss:ckc}),
where in the first case we assume~$|\alpha| - |\beta| = 1$ and in the second~$|\alpha| - |\beta| = 2$.
Composing these maps with the tautological embeddings~$\cU \hookrightarrow \DV \otimes \cO$,
$\bw2\cU \hookrightarrow \bw2\DV \otimes \cO$, and~$\Sym^2\cU \hookrightarrow \Sym^2\DV \otimes \cO$,
we obtain morphisms that we denote by
\begin{equation*}
\Sigma^\alpha\cU \xrightarrow{\ \vartheta_1\ } \DV \otimes \Sigma^\beta \cU,
\qquad 
\Sigma^\alpha\cU \xrightarrow{\ \vartheta_2\ } \bw2\DV \otimes \Sigma^\beta \cU,
\qquad 
\Sigma^\alpha\cU \xrightarrow{\ \vartheta_2^+\ } \Sym^2\DV \otimes \Sigma^\beta \cU,
\end{equation*}
respectively.
We also denote by
\begin{equation*}
\bS_k \otimes \DV \xrightarrow{\ \varkappa_1\ } \bS_{k+1},
\qquad \text{and} \qquad
\bS_k \otimes \bw2\DV \xrightarrow{\ \varkappa_2\ } \bS_{k+2},
\end{equation*}
the morphisms induced by the embeddings~$\bw{i}\DV \hookrightarrow \rcB_i(\DV)$ and the action of~$\rcB_i(\DV)$ on~$\bS_k$.
Finally, we write~$\bq_W \colon \Sym^2\DV \to \kk$ for the map induced by the quadratic form~\eqref{eq:w-form}.

\begin{theorem}[\cite{FK18}]
\label{thm:fk}
There is an exact $\Spin(\DV)$-equivariant sequence of vector bundles on~$\Gr(n,\DV)$
\begin{equation}
\label{eq:ogr-gr-spinor}
0 \to \cF_{n(n+1)/2} \to \cF_{n(n+1)/2-1} \to \dots \to \cF_1 \to \cF_0 \to j_*\cS \to 0,
\end{equation}
where
\begin{equation}
\label{eq:cfi-gr}
\cF_i \coloneqq \bigoplus_{\substack{\alpha \in \SYD_n\\[.5ex]|\alpha| + \ell_{\diag}(\alpha) = 2i}} 
\bS_{|\alpha|} \otimes \Sigma^\alpha \cU.
\end{equation}
All differentials in~\eqref{eq:ogr-gr-spinor} are linear combinations of three types of maps:
\begin{align}
\label{eq:df1}
\bS_{|\alpha|} \otimes \Sigma^\alpha \cU \xrightarrow{\ \vartheta_1\ }
\bS_{|\alpha|} \otimes \DV \otimes \Sigma^\beta \cU \xrightarrow{\ \varkappa_1\ }
\bS_{|\beta|} \otimes \Sigma^\beta \cU, 
&& \text{if~$\beta \subset \alpha$ and~$|\alpha| - |\beta| = 1$,}
\\
\label{eq:df2}
\bS_{|\alpha|} \otimes \Sigma^\alpha \cU \xrightarrow{\ \vartheta_2\ }
\bS_{|\alpha|} \otimes \bw2\DV \otimes \Sigma^\beta \cU \xrightarrow{\ \varkappa_2\ }
\bS_{|\beta|} \otimes \Sigma^\beta \cU, 
&& \text{if~$\beta \subset \alpha$ and~$|\alpha| - |\beta| = 2$,}
\\
\label{eq:df3}
\bS_{|\alpha|} \otimes \Sigma^\alpha \cU \xrightarrow{\ \vartheta_2^+\ }
\bS_{|\alpha|} \otimes \Sym^2\DV \otimes \Sigma^\beta \cU \xrightarrow{\ \bq_W\ }
\bS_{|\beta|} \otimes \Sigma^\beta \cU, 
&& \text{if~$\beta \subset \alpha$ and~$|\alpha| - |\beta| = 2$.}
\end{align}
\end{theorem}

\begin{proof}
The first claim is proved in~\cite[Proposition~3.3]{FK18};
for convenience we recall the argument.

Consider the standard full exceptional collection~$\{ \Sigma^\alpha\cU\}_{\alpha \in \YD_n}$ on~$\Gr(n,\DV)$,
the corresponding resolution of the diagonal
\begin{equation}
\label{eq:diagonal-gr}
\Big\{ 
\bw{n^2}(\cU \boxtimes \cU^\perp) \to 
\dots \to
\bw2(\cU \boxtimes \cU^\perp) \to 
\cU \boxtimes \cU^\perp \to
\cO \boxtimes \cO 
\Big\} \cong 
\Delta_*\cO
\end{equation}
on~$\Gr(n,\DV) \times \Gr(n,\DV)$,
and the Cauchy formula
\begin{equation}
\label{eq:cauchy}
\bw{i}(\cU \boxtimes \cU^\perp) = \bigoplus_{|\alpha| = i} \Sigma^\alpha\cU \boxtimes \Sigma^{\alpha^T}\cU^\perp.
\end{equation}
Let~$p_1$ and~$p_2$ denote the projections of~$\Gr(n,\DV) \times \Gr(n,\DV)$ to the factors.
Tensoring~\eqref{eq:diagonal-gr} with~$p_2^*j_*\cS$, and pushing it forward along~$p_1$, 
we obtain the hypercohomology spectral sequence
\begin{equation*}
\bE_1^{p,q} = 
\bigoplus_{|\alpha| = -p}
\rH^q(\Gr(n,\DV), \Sigma^{\alpha^T}\cU^\perp \otimes j_*\cS) \otimes \Sigma^\alpha \cU \Rightarrow
j_*\cS
\end{equation*}
On the other hand, using the projection formula and isomorphism~$j^*\cU^\perp \cong j^*\cU$, we obtain isomorphisms
\begin{equation*}
\rH^\bullet(\Gr(n,\DV), \Sigma^{\alpha^T}\cU^\perp \otimes j_*\cS) \cong
\rH^\bullet(\OGr(n,\DV), \Sigma^{\alpha^T}j^*\cU^\perp \otimes \cS) \cong
\rH^\bullet(\OGr(n,\DV), \Sigma^{\alpha^T}j^*\cU \otimes \cS).
\end{equation*}
Finally, the argument of~\cite[Lemma~3.1]{FK18} proves that
\begin{equation}
\label{eq:rh-ogr-cu-cs}
\rH^q(\OGr(n,\DV), \Sigma^{\alpha^T}j^*\cU \otimes \cS) =
\begin{cases}
\bS_{|\alpha|}, & \text{if~$\alpha \in \SYD_n$ and~$q = (|\alpha| - \ell_\diag(\alpha))/2$,}\\
0, & \text{otherwise,}
\end{cases}
\end{equation} 
and it follows that the spectral sequence becomes the exact sequence~\eqref{eq:ogr-gr-spinor}.

Since the embedding~$j$, the projections~$p_i$, the resolution of the diagonal~\eqref{eq:diagonal-gr},
and the vector bundle~$\cS$ are all $\Spin(\DV)$-equivariant,
so is the obtained exact sequence~\eqref{eq:ogr-gr-spinor}.

It remains to show that the differentials are linear combinations of the maps~\eqref{eq:df1}, \eqref{eq:df2}, and~\eqref{eq:df3}.
On the one hand, note that Borel--Bott--Weil theorem gives an isomorphism of~$\Spin(\DV)$-representations
\begin{equation*}
\Hom(\Sigma^\alpha\cU, \Sigma^\beta\cU) \cong
\begin{cases}
\Sigma^{\alpha/\beta}(\DV), & \text{if~$\beta \subset \alpha$},\\
0, & \text{if~$\beta \not\subset \alpha$},
\end{cases}
\end{equation*}
where~$\Sigma^{\alpha/\beta}$ is the Schur functor associated with the skew-diagram~$\alpha/\beta$ (see~\cite[\S2.1]{W}).
Furthermore, if~$\beta \subset \alpha$ and~$|\alpha| + \ell_{\diag}(\alpha) = |\beta| + \ell_{\diag}(\beta) + 2$
then either
\begin{itemize}
\item
$|\alpha| = |\beta| + 1$ and~$\ell_{\diag}(\alpha) = \ell_{\diag}(\beta) + 1$, or
\item
$|\alpha| = |\beta| + 2$ and~$\ell_{\diag}(\alpha) = \ell_{\diag}(\beta)$.
\end{itemize}

On the other hand, if~$|\alpha/\beta| = 1$ then~$\Sigma^{\alpha/\beta}(\DV) = \DV$,
and if~$|\alpha/\beta| = 2$ and the two boxes of~$\alpha/\beta$ occupy different rows and columns 
(which is the case if both~$\alpha$ and~$\beta$ are symmetric)
then~$\Sigma^{\alpha/\beta}(\DV) = \DV \otimes \DV$.
It remains to note that the canonical maps~$\Sigma^\alpha\cU \to W^{\otimes k} \otimes \Sigma^\beta\cU$
are the sums of maps~$\vartheta_1$, $\vartheta_2$, and~$\vartheta_2^+$,
and the only~$\Spin(\DV)$-invariant maps~$\bS_k \otimes \DV \to \bS_{k+1}$
and~$\bS_k \otimes \DV \otimes \DV \to \bS_{k+2} = \bS_k$
are induced by the maps~$\varkappa_1$, $\varkappa_2$, and~$\bq_W$, respectively.
\end{proof}

It would be interesting to understand the differentials of~$\cF_\bullet$ better (cf.~Remark~\ref{rem:ck-differentials}).

\subsection{Proof of Theorem~\ref{thm:clifford-koszul}}
\label{ss:proof-ck}

Consider the diagram
\begin{equation}
\label{eq:diagram-pi-gamma}
\vcenter{\xymatrix{
& \Sym^2V^\vee \setminus \{0\} \ar[dl]_\pi \ar[dr]^\gamma
\\
\P(\Sym^2V^\vee) &&
\Gr(n,\DV),
}}
\end{equation}
where~$\pi$ is the natural projection and the morphism~$\gamma$ is the restriction of the embedding
\begin{equation*}
\gamma \colon \Sym^2V^\vee \hookrightarrow \Gr(n,\DV),
\qquad
q \mapsto
\gamma(q) \coloneqq \Ima\left(V \xrightarrow{\ (\id,q)\ } V \oplus V^\vee = \DV\right)
\end{equation*}
that takes a quadratic form~$q \in \Sym^2V^\vee$ to its graph.
We will show that the pullback of the complex~\eqref{eq:ogr-gr-spinor} along~$\gamma$ is exact,
that it descends along~$\pi$ to an exact sequence on~$\P(\Sym^2V^\vee)$,
and the resulting exact sequence coincides with~\eqref{eq:clifford-koszul}.
We consider the obvious action of the group~$\GL(V)$ on~\eqref{eq:diagram-pi-gamma};
using the epimorphism~$\tGL(V) \to \GL(V)$ (see~\eqref{eq:tgl})
we also obtain an action of~$\tGL(V)$.

\begin{lemma}
\label{lem:diagram-equivariant}
The diagram~\eqref{eq:diagram-pi-gamma} is $\tGL(V)$-equivariant.
\end{lemma}

\begin{proof}
The morphism~$\pi$ is obviously~$\GL(V)$-equivariant.
To show that~$\gamma$ is also~$\GL(V)$-equivariant, 
recall that the action of~$\GL(V)$ on~$\DV = V \oplus V^\vee$ is given by~$g(v,f) = (gv, fg^{-1})$.
Therefore,
\begin{equation*}
g\cdot \gamma(q) = 
g \cdot \{(v, q(v)\} =
\{ (gv, q(v)g^{-1}) \} =
\{ (v, q(g^{-1}v)g^{-1}) \} =
\gamma(g^{-1}q).
\end{equation*}
It remains to use the epimorphism~$\tGL(V) \to \GL(V)$ to induce a $\tGL(V)$-structure.
\end{proof}

\begin{remark}
It may look strange that we use the $\tGL(V)$-equivariant structure instead of the $\GL(V)$-equivariant one,
but this will become very useful later when we will consider vector bundles on~$\Gr(n,\DV)$
which are $\Spin(\DV)$-equivariant and not~$\SO(\DV)$-equivariant.
\end{remark}

Note that the subspace~$\gamma(q) \subset \DV$ is canonically isomorphic to~$V$ (via the projection~$\DV \to V$).

\begin{lemma}
\label{lem:q-upsilon}
Under the canonical isomorphism~$V \cong \gamma(q)$ we have the equality 
\begin{equation}
\label{eq:q-upsilon}
\bq_W\vert_{\gamma(q)} = 2q.
\end{equation}
In particular, $\gamma(\Sym^2V^\vee \setminus \{0\}) \subset \Gr(n,\DV) \setminus \OGr(n,\DV)$.
\end{lemma}

\begin{proof}
The definition~\eqref{eq:w-form} of the form~$\bq_W$ implies that for any~$v_1,v_2 \in V$ we have
\begin{equation*}
\bq_W((v_1,q(v_1)), (v_2,q(v_2))) = q(v_2,v_1) + q(v_1,v_2) = 2q(v_1,v_2),
\end{equation*}
which proves~\eqref{eq:q-upsilon}.
The second claim also follows easily.
\end{proof}

Now we are ready to prove the theorem.

\begin{proof}[Proof of Theorem~\textup{\ref{thm:clifford-koszul}}]
An immediate consequence of Lemma~\ref{lem:q-upsilon} is that the pullback of~\eqref{eq:ogr-gr-spinor}
\begin{equation*}
\gamma^*\cF_\bullet \coloneqq 
\Big\{ \gamma^*\cF_{n(n+1)/2} \to \gamma^*\cF_{n(n+1)/2-1} \to \dots \to \gamma^*\cF_1 \to \gamma^*\cF_0 \Big\}
\end{equation*}
on the punctured affine space~$\Sym^2V^\vee \setminus \{0\}$ is exact.
Since~\eqref{eq:ogr-gr-spinor} is $\Spin(\DV)$-equivariant,
the inclusion~$\tGL(V) \hookrightarrow \Spin(\DV)$ (see~\eqref{eq:tgl})
induces on~$\gamma^*\cF_\bullet$ a $\tGL(V)$-equivariant structure.
Below, we compute the terms of~$\gamma^*\cF_\bullet$;
in other words, we describe~$\gamma^*\left(\bS_{|\alpha|} \otimes \Sigma^\alpha \cU \right)$
as~$\tGL(V)$-equivariant vector bundles.

On the one hand, we observe an isomorphism~$\gamma^*\cU \cong V \otimes \cO$, which follows from the definition of~$\gamma$, 
and which is~$\GL(V)$-equivariant, and hence~$\tGL(V)$-equivariant as well.
Applying to it the Schur functor~$\Sigma^\alpha$, 
we obtain a $\tGL(V)$-equivariant isomorphism
\begin{equation}
\label{eq:gamma-pb-cu}
\gamma^*(\Sigma^\alpha \cU) \cong \Sigma^\alpha V \otimes \cO.
\end{equation} 
On the other hand, $\bS_{|\alpha|}$ is a representation of~$\Spin(\DV)$, whose restriction to~$\tGL(V)$ is given by~\eqref{eq:bs01}.
Therefore, taking~\eqref{eq:gamma-pb-cu} into account we obtain an isomorphism
\begin{equation*}
\gamma^*\left(\bS_{|\alpha|} \otimes \Sigma^\alpha \cU \right) \cong
\left(\bigoplus_{i=0}^n \bw{2i+\epsilon(\alpha)}V\right) \otimes (\det V)^{-\frac12} \otimes \Sigma^\alpha V \otimes \cO,
\end{equation*}
where we write~$\epsilon(\alpha) \in \{0,1\}$ for the parity of~$|\alpha|$.
Note that the subgroup~$\cZ_{\Spin} \subset \tGL(V)$ defined in~\eqref{eq:cz} acts trivially on~$\Sym^2V^\vee$,
hence it acts on the complex~$\gamma^*\cF_\bullet$ fiberwise.
Moreover, the above formula shows that the action of~$\cZ_{\Spin}$ on the complex
\begin{equation*}
\gamma^*\cF_\bullet \otimes (\det V)^{\frac12}
\end{equation*}
is trivial, hence this complex 
has a $(\tGL(V)/\cZ_{\Spin})$-equivariant structure.

Now recall the subgroup~$\cZ_{\tGL} \subset \tGL(V)$ defined in~\eqref{eq:cz-tgl}
and note that~$\cZ_{\Spin} \subset \cZ_{\tGL}$ and~$\cZ_{\tGL}/\cZ_{\Spin} \cong \Gm$ acts on~$\Sym^2V^\vee$ by dilations;
therefore, there is an exact equivalence between the categories of
\begin{itemize}
\item 
$(\tGL(V)/\cZ_{\Spin})$-equivariant sheaves on~$\Sym^2V^\vee \setminus \{0\}$, and
\item 
$(\tGL(V)/\cZ_{\tGL})$-equivariant sheaves on~$\P(\Sym^2V^\vee)$.
\end{itemize}
Since~$\tGL(V)/\cZ_{\tGL} \cong \PGL(V)$, we deduce the existence 
of an exact $\PGL(V)$-equivariant sequence of vector bundles~$\cG_\bullet$ on~$\P(\Sym^2V^\vee)$ 
such that~$\pi^*\cG \cong \gamma^*\cF_\bullet \otimes (\det V)^{\frac12}$.

Now we check that the terms of~$\cG_\bullet$ are given by~\eqref{eq:ck-cgi}. 
For this just note that~$\cZ_{\tGL}/\cZ_{\Spin} \cong \Gm$ 
acts on the space~$\bw{2i+\epsilon(\alpha)}V \otimes \Sigma^\alpha V$ 
with weight~$(2i + \epsilon(\alpha) + |\alpha|)/2 = i + \lceil |\alpha|/2 \rceil$,
which means that the sheaf~$\gamma^*\left(\bS_{|\alpha|} \otimes \Sigma^\alpha \cU \right) \otimes (\det V)^{\frac12}$  
descends to
\begin{equation*}
\bigoplus_{i=0}^n \bw{2i+\epsilon(\alpha)}V \otimes \Sigma^\alpha V \otimes 
\cO\left(-i - \left\lceil \tfrac{|\alpha|}2 \right\rceil\right) \cong
\Sigma^\alpha V \otimes \cB_{-|\alpha|}
\end{equation*}
(where the second isomorphism follows from~\eqref{eq:cliff-i-again}), as required.

It remains to identify the differentials in~$\cG_\bullet$;
in particular (this is the crucial point!), to show that they are morphisms of~$\cB_0$-modules.
For this we use the description of the differentials in~$\cF_\bullet$ provided by Theorem~\ref{thm:fk}.
More precisely, we will show that the pullbacks along~$\gamma$ of the maps~\eqref{eq:df1}, \eqref{eq:df2}, and~\eqref{eq:df3}
twisted by~$(\det V)^{\frac12}$ coincide with the pullbacks along~$\pi$ of the maps~\eqref{eq:dg1}, \eqref{eq:dg2}, and~\eqref{eq:dg3};
since~$\pi^*$ is an equivalence of the corresponding equivariant categories, 
this will provide the required identification.
Finally, since the maps~\eqref{eq:dg1}, \eqref{eq:dg2}, and~\eqref{eq:dg3} are morphisms of~$\cB_0$-modules,
it will prove that~\eqref{eq:clifford-koszul} is a complex of $\cB_0$-modules and complete the proof of the theorem.

So, we fix a pair~$\alpha$, $\beta$ of symmetric Young diagrams 
such that~$\beta \subset \alpha$ and~$|\alpha| - |\beta| = k \in \{1,2\}$
and a quadratic form~$q \in \Sym^2V^\vee \setminus \{0\}$.
Using the isomorphism~\eqref{eq:gamma-pb-cu}
we see that the pullbacks along~$\gamma$ of the maps~\eqref{eq:df1} and~\eqref{eq:df2} coincide with the compositions
\begin{equation*}
\Sigma^\alpha V \otimes \bS_{|\alpha|} \xrightarrow{\ \theta_k\ }
\Sigma^\beta V \otimes \bw{k} V \otimes \bS_{|\alpha|} \xrightarrow{\ \wedge^k(\id,q)\ }
\Sigma^\beta V \otimes \bw{k} \DV \otimes \bS_{|\alpha|} \xrightarrow{\ \varkappa_k\ }
\Sigma^\beta V \otimes \bS_{|\beta|}
\end{equation*}
Similarly, the pullbacks along~$\pi$ of the maps~\eqref{eq:dg1} and~\eqref{eq:dg2} coincide with the maps
\begin{equation*}
\Sigma^\alpha V \otimes \pi^*\cB_{-|\alpha|} \xrightarrow{\ \theta_k\ }
\Sigma^\beta V \otimes \bw{k} V \otimes \pi^*\cB_{-|\alpha|} \xrightarrow{\ \kappa_k\ }
\Sigma^\beta V \otimes \pi^*\cB_{-|\beta|}.
\end{equation*}
Using~\eqref{eq:cliff-i-again} as before to identify~$\gamma^*\bS_{|\alpha|} \otimes (\det V)^{\frac12}$ with~$\pi^*\cB_{-|\alpha|}$
and~$\gamma^*\bS_{|\beta|} \otimes (\det V)^{\frac12}$ with~$\pi^*\cB_{-|\beta|}$
we see that to prove the required identification of~\eqref{eq:dg1} and~\eqref{eq:df1}, it is enough to show that the composition
\begin{equation*}
\bw{\bullet}V \otimes V \xrightarrow{\ \id \otimes (\id,q)\ }
\bw{\bullet}V \otimes (V \oplus V^\vee) \xrightarrow{\qquad\quad\ }
\bw{\bullet}V,
\end{equation*}
where the second map is given by wedge product with~$V$ and convolution with~$V^\vee$,
coincides with the Clifford multiplication in~$\cB(V,q)$.
But this is clear, because this composition acts as
\begin{equation*}
(v_1 \wedge \dots \wedge v_s) \otimes v \mapsto
(v_1 \wedge \dots \wedge v_s) \otimes (v,q(v)) \mapsto
v_1 \wedge \dots \wedge v_s \wedge v +
\sum_{i=1}^s (-1)^{i-s} q(v,v_i) v_1 \wedge \dots \wedge \widehat{v_i} \wedge \dots \wedge v_s,
\end{equation*}
which is exactly the Clifford multiplication (compare with~\eqref{eq:clifford-multiplication}).

The identification of~\eqref{eq:dg2} and~\eqref{eq:df2} is proved analogously
(we can simplify the verification a little bit by using the associativity of the Clifford algebra).

Finally, to identify~\eqref{eq:dg3} and~\eqref{eq:df3}
we note that the pullback along~$\gamma$ of~\eqref{eq:df3} is the composition
\begin{equation*}
\Sigma^\alpha V \otimes \bS_{|\alpha|} \xrightarrow{\ \theta_2^+\ }
\Sigma^\beta V \otimes \Sym^2 V \otimes \bS_{|\alpha|} \xrightarrow{\ \Sym^2(\id,q)\ }
\Sigma^\beta V \otimes \Sym^2 \DV \otimes \bS_{|\alpha|} \xrightarrow{\ \bq_W\ }
\Sigma^\beta V \otimes \bS_{|\beta|}
\end{equation*}
and the pullback along~$\pi$ of the map~\eqref{eq:dg3} is the composition
\begin{equation*}
\Sigma^\alpha V \otimes \pi^*\cB_{-|\alpha|} \xrightarrow{\ \theta_2^+\ }
\Sigma^\beta V \otimes \Sym^2 V \otimes \pi^*\cB_{-|\alpha|} \xrightarrow{\ \bq_W\ }
\Sigma^\beta V \otimes \pi^*\cB_{-|\beta|}.
\end{equation*}
Now, the required identification follows from~\eqref{eq:q-upsilon}.
\end{proof}

\section{Homological properties of Clifford spaces}

In this section we study homological properties of the coordinate algebras~$\bB_U$ of Clifford spaces.
In~\S\ref{ss:koszul} we show that the algebra~$\bB_U$ is Koszul if the Clifford space is minimal,
and in~\S\ref{ss:as-regularity} we show that~$\bB_U$ is Artin--Schelter regular for any Clifford space.

\subsection{The Koszul property}
\label{ss:koszul}

Let $U \subset \Sym^2(V^\vee)$ be a subspace with~$\dim(U) = \dim(V) = n$ satisfying~\eqref{eq:empty},
and let~$(\P(U),\cB_0(V))$ be the corresponding minimal Clifford space.
The results of Section~\ref{sec:cs} (Theorems~\ref{thm:clifford-helix} and~\ref{thm:clifford-algebra})
in this case give the full and strong exceptional collection
\begin{equation}
\label{eq:db-min-cs}
\Db(\P(U),\cB_0(V))) = \Big\langle \cB_{1-n}(V), \dots, \cB_{-1}(V), \cB_0(V) \Big\rangle,
\end{equation}
the 1-periodic strong helix~$\{ \cB_i(V) \}_{i \in \ZZ}$,
and the graded algebra~$\bB_U = \bigoplus_{i=0}^\infty \Hom_{\cB_0(V)}(\cB_0(V), \cB_i(V))$ such that
\begin{equation*}
\coh(\P(U),\cB_0(V))) \simeq \qgr(\bB_U).
\end{equation*}
In this subsection we show that~$(\P(U),\cB_0(V))$ is an example of noncommutative projective space.
We start with a simple but useful example.

\begin{example}
\label{ex:toric}
Let~$\rT \subset \GL(V)$ be a maximal torus and let
\begin{equation}
\label{eq:w-toric}
U_\rT \coloneqq \langle f_1^2, f_2^2, \dots, f_n^2 \rangle \subset \Sym^2(V^\vee),
\end{equation}
where $f_1,\dots,f_n$ is a $\rT$-invariant basis of~$V^\vee$.
Clearly, \eqref{eq:empty} holds for this subspace.
The corresponding algebra~$\bB_{U_\rT}$ is known as a {\sf skew polynomial algebra}.
\end{example}

Recall that the Hilbert series for a graded algebra~$\bA$ is the formal power series
\begin{equation*}
\bh_\bA(z) \coloneqq \sum_{i=0}^\infty \dim(\bA_i)z^i.
\end{equation*}

\begin{lemma}
\label{lem:helix-min-cs}
Let~$(\P(U),\cB_0(V))$ be a minimal Clifford space.
The Hilbert series of its coordinate algebra~$\bB_U$
equals the Hilbert series of the polynomial algebra in~$n$ variables, i.e.,
\begin{equation*}
\bh_{\bB_U}(z) = \frac1{(1-z)^n}.
\end{equation*}
In particular, the algebras~$\bB_U$ for all~$n$-dimensional subspaces~$U \subset \Sym^2(V^\vee)$ satisfying~\eqref{eq:empty}
form a flat deformation family.
\end{lemma}

\begin{proof}
By definition of the algebra~$\bB_U$, isomorphism~\eqref{eq:ext-clifford} and Lemma~\ref{lem:cohomology-clifford}, we have
\begin{equation}
\label{eq:algebra-ncp}
\bB_{U,i} =
\rH^0(\P(U), \cB_i(V))) \cong
\bw{i}V \oplus \Big(\bw{i-2}V \otimes U^\vee \Big) \oplus \Big(\bw{i-4}V \otimes \Sym^2(U^\vee)\Big) \oplus \dots.
\end{equation}
Using a simple combinatorial identity we compute~$\dim(\bB_{U,i})$ as
\begin{equation}
\label{eq:binomial}
\sum_{s \ge 0} \tbinom{n}{i-2s} \cdot \tbinom{n+s-1}{s} = \tbinom{n+i-1}{i} = \dim \Sym^i(\kk^n),
\end{equation}
and it follows that~$\bh_{\bB_U}(z) = (1 - z)^{-n}$.
Since the set of all subspaces~$U \subset \Sym^2V^\vee$ of dimension~$n$ satisfying the assumption~\eqref{eq:empty}
is an open subset of the Grassmannian~$\Gr(n, \Sym^2V^\vee)$ (in particular, it is connected),
we conclude that all these algebras form a flat deformation family.
\end{proof}

\begin{remark}
One way to prove the identity~\eqref{eq:binomial} is by observing
that the right-hand side is the number of monomials of degree~$i$ in variables~$x_1$, \dots, $x_n$,
and each monomial can be uniquely written as a product of a square-free monomial of degree~$i - 2s$
and a monomial of degree~$s$ in~$x_i^2$ (cf.\ the proof of Theorem~\ref{thm:intro-bb}).
\end{remark}

To prove that the algebras~$\bB_U$ have nice homological properties we use a notion from~\cite{BP93},
which we slightly reformulate for convenience.

\begin{definition}
\label{def:helix-geometric}
A {\sf geometric helix} in a triangulated category~$\cT$
is an infinite sequence~$\{\cE_i\}_{i \in \ZZ}$ of exceptional objects in~$\cT$
such that there is an integer~$n$ with the following properties:
\begin{alenumerate}
\item
\label{it:helix-fec}
for each~$i \in \ZZ$ the collection~$(\cE_i, \cE_{i+1}, \dots, \cE_{i+n-1})$
is a full exceptional collection in~$\cT$,
\item
\label{it:helix-serre}
for each~$i$ one has~$\Ext^\bullet(\cE_i,\cE_{i-n}) \cong \kk[1-n]$, and
\item
\label{it:helix-geometric}
$\Ext^p(\cE_i,\cE_j) = 0$ for all~$p \ne 0$ and~$i \le j$.
\end{alenumerate}
\end{definition}

\begin{remark}
\label{rem:helix}
In fact, hypotheses~\ref{it:helix-fec} and~\ref{it:helix-serre}
are equivalent to the definition of a helix in~\cite[\S1]{BP93},
while hypothesis~\ref{it:helix-geometric} is the defining property of geometric helices.
Note that the definition of a helix that was used in~\S\ref{sec:cs} and~\S\ref{sec:cms}
only included hypothesis~\ref{it:helix-fec},
while~\ref{it:helix-geometric} was taken as the definition of a strong helix.
\end{remark}

\begin{theorem}
\label{thm:bbu}
If~$(\P(U), \cB_0(V)))$ is a minimal Clifford space, the helix~$\{\cB_i(V))\}_{i \in \ZZ}$ is a geometric helix.
Moreover, the graded algebra~$\bB_U$ is Koszul and there is an exact sequence of~$\bB_U$-modules
\begin{equation}
\label{eq:mck}
0 \to
\bB_U(-n) \to
\bB_U(1-n)^{\oplus n} \to
\dots \to
\bB_U(-2)^{\oplus \binom{n}{2}} \to
\bB_U(-1)^{\oplus n} \to
\bB_U \to
\kk \to
0.
\end{equation}
\end{theorem}

\begin{proof}
First, we note that hypotheses~\ref{it:helix-fec} and~\ref{it:helix-geometric}
for the sequence~$\{\cB_i(V))\}_{i \in \ZZ}$ were verified in Theorem~\ref{thm:clifford-helix},
and hypothesis~\ref{it:helix-serre} follows immediately from isomorphism~\eqref{eq:ext-clifford} and Lemma~\ref{lem:cohomology-clifford}.
This proves that~$\{\cB_i(V))\}_{i \in \ZZ}$ is a geometric helix.

Note that Theorem~\ref{thm:clifford-helix} also proves that the helix is 1-periodic.
Therefore, by~\cite[Theorem~4.2 and Proposition~4.1]{BP93} the algebra~$\bB_U$ is Koszul,
its dual algebra~$\bB_U^!$ is Frobenius of index~$n$, and the simple $\bB_U$-module~$\kk$ has a free resolution
\begin{equation*}
0 \to \bB^!_{U,n} \otimes \bB_U(-n) \to
\bB^!_{U,n-1} \otimes \bB_U(1-n) \to
\dots \to
\bB^!_{U,2} \otimes \bB_U(-2) \to
\bB^!_{U,1} \otimes \bB_U(-1) \to
\bB_U \to \kk \to 0,
\end{equation*}
and it remains to compute the dimensions of the graded components~$\bB^!_{U,i}$ of the dual algebra.
For this we use the relation between the Hilbert series of the algebras~$\bB_U$ and~$\bB_U^!$
\begin{equation*}
\bh_{\bB_U}(z) \cdot \bh_{\bB^!_U}(-z) = 1,
\end{equation*}
see~\cite[Corollary~2.2.2]{PP05}, which implies that~$\bh_{\bB^!_U}(z) = (1 + z)^n$,
and therefore~$\dim(\bB^!_{U,i}) = \binom{n}{i}$.
We conclude that the resolution of the simple module takes the form~\eqref{eq:mck}.
\end{proof}

\begin{corollary}
\label{cor:bbu}
If~$(\P(U), \cB_0(V))$ is a minimal Clifford space then
\begin{equation*}
\bB_U \cong \bT(V) / \langle U^\perp \rangle,
\end{equation*}
where~$U^\perp \coloneqq \Ker(\Sym^2V \to U^\vee) \subset \Sym^2V \subset V \otimes V$
and~$\langle U^\perp \rangle$ is the two-sided ideal in the tensor algebra~$\bT(V)$
generated by~$U^\perp \subset V \otimes V$.
In other words, $\bB_U$ is the graded Clifford algebra.
\end{corollary}

\begin{proof}
The multiplication~$V \otimes V = \bB_{U,1} \otimes \bB_{U,1} \to \bB_{U,2} = \bw2V \oplus U^\vee$
is given by
\begin{equation*}
v_1 \otimes v_2 \mapsto v_1 \wedge v_2 + \bq_U(v_1 \cdot v_2),
\end{equation*}
where~$\bq_U \colon \Sym^2V \to U^\vee$ is the dual map to the embedding~$U \hookrightarrow \Sym^2V^\vee$.
The restriction of the multiplication to~$\bw2V \subset V \otimes V$ is an isomorphism onto~$\bw2V \subset \bB_{U,2}$,
and its restriction to~$\Sym^2V \subset V \otimes V$ is the map~$\bq_U$.
Therefore, the space of quadratic relations for~$\bB_U$ is~$\Ker(\bq_U) = U^\perp$,
and since~$\bB_U$ is Koszul, and in particular quadratic,
it follows that~$\bB_U \cong \bT(V) / \langle U^\perp \rangle$,
which coincides with the definition of the graded Clifford algebra.
\end{proof}

\begin{corollary}
\label{cor:bbu-generation}
If~$(\P(U), \cB_0(V))$ is a minimal Clifford space
then~$\bB_U$ is a flat deformation of the polynomial algebra~$\Sym^\bullet(V)$
and~$\bB_U^!$ is a flat deformation of the exterior algebra~$\bw\bullet V^\vee$.
\end{corollary}

\begin{proof}
By Lemma~\ref{lem:helix-min-cs} all algebras~$\bB_U$ are deformation equivalent,
hence it is enough to check that the skew polynomial algebra algebra~$\bB_{U_\rT}$
(where~$U_\rT$ is defined in~\eqref{eq:w-toric})
is a flat deformation of~$\Sym^\bullet(V)$.
To this end, we choose a basis~$\{v_1,v_2,\dots,v_n\}$ in~$V$
and consider the family of quadratic algebras
\begin{equation}
\label{eq:utq}
\bA_{\rT,q} \coloneqq \bT(V) / \langle \rR_{\rT,q} \rangle,
\qquad
\rR_{\rT,q} \coloneqq \langle v_i \otimes v_j - q \cdot v_j \otimes v_i \rangle_{i,j = 1}^n \subset V \otimes V
\end{equation}
parameterized by~$q \in \AA^1$.
It is easy to see that for each of this algebras the graded component~$(\bA_{\rT,q})_i$
has a basis formed by the monomials~$v_1^{i_1}\cdot v_2^{i_2} \cdot \dots \cdot v_n^{i_n}$, where~$\sum i_k = i$.
Therefore, $\dim(\bA_{\rT,q})_i = \binom{n+i-1}{i}$ is constant, hence~$\bA_{\rT,q}$ is a flat family of graded algebras.
Finally, $\bA_{\rT,1} \cong \Sym^2(V)$, while~$\bA_{\rT,-1} \cong \bB_{U_\rT}$, hence the first claim.

Similarly, if~$\{f_1,f_2,\dots,f_n\}$ is a basis in~$V^\vee$ then
\begin{equation*}
\bA_{\rT,q}^! \cong \bT(V^\vee) / \langle \rR_{\rT,q}^\perp \rangle,
\qquad
\rR_{\rT,q}^! \coloneqq
\langle f_i^2,\ f_i \otimes f_j + q \cdot f_j \otimes f_i \rangle_{i,j = 1}^n \subset V^\vee \otimes V^\vee,
\end{equation*}
and it is clear that~$(\bA_{\rT,q}^!)_i$ has a basis formed by the monomials~$f_{j_1}\cdot f_{j_2} \cdot \dots \cdot f_{j_i}$
where~$j_1 < j_2 < \dots < j_i$.
It follows that~$\dim(\bA_{\rT,q}^!)_i = \binom{n}{i}$ is constant,
and therefore~$\bA_{\rT,q}^!$ is a flat family of graded algebras
with~$\bA_{\rT,1}^! \cong \bw\bullet(V^\vee)$ and~$\bA_{\rT,-1}^! \cong \bB_{U_\rT}^!$, hence the second claim.
\end{proof}

\subsection{Artin--Schelter regularity}
\label{ss:as-regularity}

Recall that a graded algebra~$\bA = \bigoplus_{i = 0}^\infty \bA_i$ with~$\bA_0 = \kk$
is called {\sf Artin--Schelter regular of dimension~$d$} (see~\cite{AS87})
if every graded $\bA$-module has projective dimension at most~$d$,
Gelfand--Kirillov dimension of~$\bA$ is finite,
and~$\bA$ has the Gorenstein property
\begin{equation*}
\Ext^\bullet_\bA(\kk, \bA) \cong \kk(-\ell)[-d].
\end{equation*}
for some~$\ell \in \ZZ$ (we will say that~$\ell$ is the {\sf Gorenstein degree} of~$\bA$).

Artin--Schelter regularity of the algebra~$\bB_U$ for a minimal Clifford space was proved in~\cite[Proposition~7]{LB95}.
Below, we show that this property holds for all Clifford spaces.

\begin{proposition}
\label{prop:as-regularity}
If~$U \subset \Sym^2(V^\vee)$ is a $k$-dimensional space of quadrics satisfying the assumption~\eqref{eq:empty}
then~$\bB_U$ is Artin--Schelter regular of dimension~$k$ and Gorenstein degree~$2k - n$.
Moreover, the Hilbert series of~$\bB_U$ is
\begin{equation*}
\bh_{\bB_U}(z) = \frac1{(1-z)^{n}(1-z^2)^{k-n}}.
\end{equation*}
\end{proposition}

\begin{proof}
We prove that the simple~$\bB_U$-module~$\kk$ has a free graded resolution
\begin{equation}
\label{eq:self-dual}
0 \to F_k \to F_{k-1} \to \dots \to F_1 \to F_0 \to \kk \to 0,
\end{equation}
that has the self-duality property~$\Hom_{\bB_U}(F_{\bullet}, \bB_U) \cong F_{k-\bullet}(2k-n)$.
Note that the assumption~\eqref{eq:empty} implies that~$k \ge n$.
Let~$U_0 \subset U$ be a subspace of dimension~$n$ such that the assumption~\eqref{eq:empty} holds for~$U_0$,
so that~$(\P(U_0),\cB_0(V))$ is a minimal Clifford space.
Let~$m \coloneqq k - n$, choose a basis~$\xi_1,\dots,\xi_m$ of the vector space~$\Ker(U^\vee \to U_0^\vee)$,
and consider the Koszul resolution
\begin{equation*}
0 \to \cO_{\P(U)}(-m) \xrightarrow{\,\xi\,}
\cO_{\P(U)}(1-m)^{\oplus m} \xrightarrow{\,\xi\,}
\dots \xrightarrow{\,\xi\,}
\cO_{\P(U)}(-2)^{\oplus \binom{m}2} \xrightarrow{\,\xi\,}
\cO_{\P(U)}(-1)^{\oplus m} \xrightarrow{\,\xi\,}
\cO_{\P(U)} \to
\cO_{\P(U_0)} \to 0,
\end{equation*}
where the maps are given by wedge product with~$\xi = (\xi_1,\dots,\xi_m)$.
Tensoring it with~$\cB_0(V)$, considering the corresponding graded $\bB_U$-modules,
and taking into account that higher cohomology do not show up by Lemma~\ref{lem:cohomology-clifford},
we obtain an exact sequence
\begin{equation}
\label{eq:koszul-u}
0 \to
\bB_U(-2m) \xrightarrow{\,\xi\,}
\bB_U(2-2m)^{\oplus m} \xrightarrow{\,\xi\,}
\dots \xrightarrow{\,\xi\,}
\bB_U(-4)^{\oplus \binom{m}2} \xrightarrow{\,\xi\,}
\bB_U(-2)^{\oplus m} \xrightarrow{\,\xi\,}
\bB_U \to
\bB_{U_0} \to 0.
\end{equation}
On the other hand, consider the resolution~\eqref{eq:mck} of the simple module~$\kk$ over~$\bB_{U_0}$:
\begin{equation*}
0 \to
\bB_{U_0}(-n) \to
\bB_{U_0}(1-n)^{\oplus n} \to
\dots \to
\bB_{U_0}(-2)^{\oplus \binom{n}{2}} \to
\bB_{U_0}(-1)^{\oplus n} \to
\bB_{U_0} \to
\kk \to
0.
\end{equation*}
Replacing in it each term~$\bB_{U_0}(-i)^{\oplus \binom{n}i}$
by the direct sum of~$\tbinom{n}{i}$ copies of~\eqref{eq:koszul-u} twisted by~$-i$,
lifting the differentials of~\eqref{eq:mck} to a bicomplex structure,
and taking its totalization, we obtain a free resolution of~$\kk$ of length~$n + m = k$.
Since~\eqref{eq:mck} is self dual up to twist by~$n$ and~\eqref{eq:koszul-u} is self-dual up to twist by~$2m$,
we conclude that the obtained free resolution of~$\kk$ is self-dual up to twist by~$n + 2m = n + 2(k - n) = 2k - n$, as required.

Now we use resolution~\eqref{eq:self-dual} to compute~$\Ext^\bullet(\kk,\bB_U)$:
the self-duality property of~\eqref{eq:self-dual} implies that
\begin{equation*}
\Ext^\bullet_{\bB_U}(\kk, \bB_U) \cong \kk(n - 2k)[-k],
\end{equation*}
hence~$\bB_U$ has the required Gorenstein property.
On the other hand, \eqref{eq:self-dual} shows that the projective dimension of~$\kk$ equals~$k$,
hence the global dimension of~$\bB_U$ is also equal to~$k$.
Furthermore, using~\eqref{eq:koszul-u} we find~$\bh_{\bB_U}(z) \cdot (1 - z^2)^m = \bh_{\bB_{U_0}}(z)$,
and taking Lemma~\ref{lem:helix-min-cs} into account, we conclude that
\begin{equation*}
\bh_{\bB_U}(z) =
\frac{\bh_{\bB_{U_0}}(z)}{(1 - z^2)^m} =
\frac1{(1-z)^{n}(1-z^2)^m},
\end{equation*}
as required.
Finally, using the Hilbert series it is easy to see that the Gelfand--Kirillov dimension of~$\bB_U$ is finite,
and summarizing, we see that~$\bB_U$ is Artin--Schelter regular of dimension~$k$.
\end{proof}

\bibliographystyle{alpha}
\bibliography{clifford.bib}

\end{document}